\documentclass{amsart}
\usepackage[utf8]{inputenc}
\usepackage{mathtools}
\usepackage{amssymb}
\usepackage{amsthm}
\usepackage{enumitem}
\usepackage{hyperref}

\setlist[enumerate]{label=\rm{(\arabic*)}, ref=(\arabic*)}

\DeclareMathOperator{\St}{St}
\DeclareMathOperator{\Aut}{Aut}
\DeclareMathOperator{\rist}{Rist}
\DeclareMathOperator{\Srist}{SRist}

\newcommand*{\N}{\mathbf{N}}
\newcommand*{\Z}{\mathbf{Z}}
\newcommand*{\PP}{\mathcal{P}}
\newcommand*{\VV}{\mathcal{V}}
\newcommand*{\C}{\mathcal{C}}
\newcommand*{\Grig}{\mathcal{G}}

\newtheorem{thm}{Theorem}[section]
\newtheorem{lemma}[thm]{Lemma}
\newtheorem{prop}[thm]{Proposition}
\newtheorem{cor}[thm]{Corollary}

\theoremstyle{definition}
\newtheorem{qu}[thm]{Question}
\newtheorem{defn}[thm]{Definition}
\newtheorem{rem}[thm]{Remark}

\newtheorem{example}[thm]{Example}

\title[Structure of finitely generated subgroups of branch groups]{On the structure of finitely generated subgroups of branch groups}
\author[D. Francoeur]{Dominik Francoeur}
\address{Dominik Francoeur, Department of Mathematics, Universidad Autónoma de Madrid, Campus Cantoblanco UAM, Madrid, Spain; \texttt{dominik.francoeur@uam.es}}
\author[R. Grigorchuk]{Rostislav Grigorchuk}
\address{Rostislav Grigorchuk, Department of Mathematics, Texas A\&M University, 77843~College Station, U.S.A.; \texttt{grigorch@math.tamu.edu}}
\author[P.-H. Leemann]{Paul-Henry Leemann}
\address{Paul-Henry Leemann, Department of Pure Mathematics, Xi'an Jiaotong--Liverpool University, Suzhou, P.R. China; \texttt{PaulHenry.Leemann@xjtlu.edu.cn}}
\author[T. Nagnibeda]{Tatiana Nagnibeda}
\address{Tatiana Nagnibeda, Section de math\'ematiques, Universit\'e de Gen\`eve, 1205~Gen\`eve, Switzerland; \texttt{Tatiana.Nagnibeda@unige.ch}}
\date{\today}

\begin{document}
\begin{abstract}
Motivated by the study of profinite topology in branch groups, we prove a structural result about finitely generated subgroups in branch groups. More precisely, we show that finitely generated subgroups of a branch group with the subgroup induction property have a block structure, which roughly means that, up to a finite index, they are products of finite index subgroups, embedded in the group in a way that is coherent with its branch action on the rooted tree.
\end{abstract}

\maketitle

\section{Introduction}
 
\emph{Profinite topology} is a natural topology to consider on a group, with a basis given by all finite index subgroups (and their cosets).
It is important to understand closed subgroups, and more generally closed subsets, for this topology.
Such a study naturally begins with the question: is the trivial subgroup closed? The answer is positive exactly for the residually finite groups, in which case all finite subgroups are closed. The next step is to study closure properties of finitely generated subgroups. A group is said to be \emph{subgroup separable} or \emph{LERF} (which stands for locally extensively residually finite) if all of its finitely generated subgroups are closed in the profinite topology. 
The class of LERF groups contains finite groups; finitely generated abelian groups; virtually polycyclic groups: in fact, all subgroups in these groups are closed in the profinite topology~\cite{Malcev}. By a celebrated theorem of Marshall Hall and it subsequent extensions, it also contains finitely generated free groups; surface groups and more generally, limit groups~\cite{MR2399104}. The LERF property is also known to hold for certain, but not for all, right-angled Artin groups; for free metabelian groups, but not for free solvable groups of derived length bigger than~$2$; and also for some other examples (see discussions in~\cite{MR2399104,Coulbois}); it remains in general not well understood. 
A stronger condition called the \emph{Ribes-Zalesskii property} asks that all  products of finitely many finitely generated subgroups be closed in the profinite topology~\cite{MR1190361}. Besides finite groups and finitely generated abelian groups, it has so far been shown to hold for free groups, limit groups, Kleinian groups and a few additional examples, and to be closed under the operations of subgroups, finite index supergroups and free products, see~\cite{Coulbois,MM} and the references therein.

The study of LERF property for branch groups was initiated by the second author and Wilson in~\cite{GrigorchukWilson03}.
Further examples of branch groups with the LERF properties can be found in~\cite{Garrido16,FrancoeurLeemann20}.
One of the motivations of the present work is to understand the structure of finitely generated subgroups in branch groups in view of better understanding the LERF property, the Ribes-Zalesskii property and more generally, the nature of subsets closed in the profinite topology, for this class of groups. 

\emph{Branch groups} are groups acting faithfully on a spherically homogeneous rooted tree in a fashion similar to the action of the full automorphism group of the tree, see Definition~\ref{defn:branch}.
They constitute a class of residually finite groups with a rich subgroup structure which contains many exotic examples of groups including the first example of a group of intermediate growth, constructed by the second author in~\cite{Grigorchuk80} and now known as the \emph{(first) Grigorchuk group} (throughout  the article  this  group  will  be  denoted  by  $\Grig$; see Example~\ref{example:GrigorchukGroup} for a definition). Branch groups appear naturally as one of three types of \emph{just infinite groups} (see~\cite{Wilson} and~\cite{GrigorchukWilson03}), that is, infinite groups with only finite proper quotients. Many examples of branch groups can be found among self-similar groups. Recall that a group of automorphisms of a regular rooted tree is \emph{self-similar} if the sections of its elements to the subtrees rooted at vertices of the tree are again elements of this group, see Definition~\ref{defn:selfsim}. If moreover the section maps are surjective for every vertex of the tree, the group is said to be \emph{self-replicating}.

The lattice of subgroups in branch groups is being extensively studied. In particular, such questions as the \emph{congruence subgroup property} (see Remark~\ref{rem:CSP}) or the nature of maximal subgroups received considerable attention.
Concerning the latter, Pervova showed~\cite{MR1841763} that the maximal subgroups of $\Grig$ are necessarily of finite index, moreover they are all of index 2 and hence are normal.
Examples are known of  branch groups with maximal subgroups of infinite index~\cite{MR2727305,MR3886188}, but these seem rather to be an exception, see~\cite{Francoeur2020}. In contrast, the variety of \emph{weakly maximal subgroups} (i.e., subgroups that are maximal among infinite index subgroups) is very rich in any branch group~\cite{MR3478865,Leemann20}. Under certain assumptions, weakly maximal subgroups, whether they are finitely generated or not, are closed in  the profinite topology~\cite{GrigorchukLeemannNagnibeda21,FrancoeurLeemann20}.
Another motivation for our work is to better understand weakly maximal subgroups in branch groups.

First ideas for the present work appeared in the second and the last authors' short note~\cite{GN-Oberwolfach}, where the notion of a \emph{block subgroup} of a self-similar group was introduced.
Roughly speaking, a block subgroup of a self-similar group $G$ is a product of copies of some finite index subgroups of $G$, some of them embedded diagonally. See Figure~\ref{fig:blockintro} for an illustration of this notion and Definitions~\ref{defn:diag} and~\ref{defn:Block} for a formal definition.
The definition in~\cite{GN-Oberwolfach} was inspired by Pervova's explicit example of a finitely generated weakly maximal subgroup of $\Grig$, see~\cite{MR2893544}. The only examples of weakly maximal subgroups of $\Grig$ known previously were the so-called \emph{parabolic subgroups}, i.e., the stabilisers of infinite rays in the infinite complete binary tree, which are all not finitely generated.
\begin{figure}[htbp]
\includegraphics{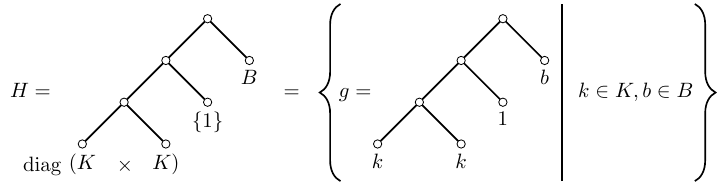}
\caption{A block subgroup of a self-similar group $G$. Here $B$ and $K$ are finite index subgroups of $G$.}
\label{fig:blockintro}
\end{figure}

Theorem 1 announced in~\cite{GN-Oberwolfach} claims that any finitely generated subgroup of $\Grig$ is virtually a block subgroup. 
In this paper, we prove this result for a larger family of self-similar branch groups.
This family is defined by two conditions. The first one is called the \emph{subgroup induction property} (see Definition~\ref{defn:SIP}). It first appeared in the second author and Wilson's proof of the LERF property for the Grigorchuk group $\Grig$~\cite{GrigorchukWilson03}, and it is also known to hold for all the so called torsion \emph{GGS groups}~\cite{FrancoeurLeemann20}, see Example~\ref{example:GGS}. The second condition is what we call \emph{tree-primitivity}. Roughly, an action of a group of automorphisms of a rooted tree is tree-primitive if it is as primitive as it can be; given that an action on a rooted tree can never be primitive, as it must preserve the partition of the tree into levels; see Definition~\ref{defn:treeprim}.

\begin{thm}[Theorem~\ref{thm:BlockSubgroups}]\label{thm:A}
Let $G$ be a finitely generated self-replicating branch group with the subgroup induction property acting tree-primitively on a regular rooted tree, and let $H\leq G$ be a finitely generated subgroup of $G$. Then, $H$ is virtually a block subgroup.
\end{thm}

As a structural result about finitely generated subgroups, Theorem~\ref{thm:A} allows a better understanding of the lattice of subgroups of a self-similar branch group. In particular, it completes the description of all weakly maximal subgroups of the Grigorchuk group and of the torsion GGS groups, see~\cite{Leemann20}.

If $G$ is \emph{regular branch} (see Definition~\ref{defn:regbranch}) and satisfies a weak form of the congruence subgroup property, then we prove a more precise result.
\begin{thm}[Theorem~\ref{thm:RegularBlockSubgroup}]\label{thm:B}
Let $G$ be a finitely generated self-replicating regular branch group with trivial branch kernel satisfying the subgroup induction property acting tree-primitively on a regular rooted tree.
Let $K$ be its maximal branching subgroup (see Definition~\ref{defn:MaxBranchingSubgroup} and Corollary~\ref{cor:MaximalBranchingSubgroup}).
Then, every finitely generated subgroup $H\leq G$ is virtually a regular block subgroup over $K$ (i.e., it admits a finite index subgroup that is a regular block subgroup over $K$).
\end{thm}
 Partial results towards Theorem~\ref{thm:A} were obtained in~\cite{GrigorchukLeemannNagnibeda21}, however the condition of induction was weaker there and the proof was not complete, see Remark~\ref{rem:WeakStrong} below.
Theorem~1 in~\cite{GN-Oberwolfach} also claims that for the specific case of $\Grig$, the block structure can be found algorithmically given the generators of the finitely generated subgroup $H$.
This question also makes sense in our generality, but we will not address it in this paper.

\subsection*{Plan of the article} The next section contains definitions and preliminaries on self-similar groups and on branch groups.
Section~\ref{sec:normal} studies normal and almost normal subgroups  of branch groups, as they turn out to be very important in the study of finitely generated subgroups.
In Section~\ref{section:TreePrimitivity} we define and study tree-primitive actions of groups of automorphisms of rooted trees.
It is in Section~\ref{section:BlockSubgroupsAndSIP} that we finally turn our attention to the subgroup induction property and to the block subgroups.
The last two sections are devoted to the proofs of Theorem~\ref{thm:A} and Theorem~\ref{thm:B}.

\subsection*{Acknowledgements}
The first and the third authors thank Anita Thillaisundaram for a useful discussion and references concerning maximal branch subgroups of GGS groups.
The authors would like to thank the anonymous referee for a careful reading and constructive criticism of a previous version of this paper.
The authors acknowledge support of the Swiss NSF grant 200020-20040.
The first author acknowledges the support of the Leverhulme Trust Research Project Grant RPG-2022-025.
The second author is supported by Travel Support for Mathematicians grant MP-TSM-00002045 from Simons Foundation.
The third author was supported by RDF-23-01-045 \emph{Groups acting on rooted trees and their subgroups} of XJTLU.

\section{Preliminaries on branch groups}

\subsection{The poset of words over a finite alphabet and branch groups}

For the rest of the article, let $X$ be a finite set. For $n\in \N$, we denote by $X^n$ the set of words of length $n$ in the alphabet $X$, i.e.\ the set of maps from $\{1,\dots, n\}$ to $X$. By convention, the set $X^0$ contains only one element, the empty word, which we denote by $\epsilon$. Let $X^*=\bigcup_{n\in \N}X^n$ be the set of all finite words on the alphabet $X$. Then, for any $v\in X^*$, there exists a unique $n\in \N$ such that $v\in X^n$. We call this $n$ the \emph{length} of $v$ and write $|v|=n$.

When endowed with the operation of concatenation of words, $X^*$ becomes the free monoid on the set $X$. This free monoid structure allows one to define a partial order on $X^*$ as follows. Given two words $u,v\in X^*$, we will say that $u\leq v$ if there exists $w\in X^*$ such that $uw=v$. If $u\leq v$, we say that $u$ is an \emph{ancestor} of $v$ and that $v$ is a \emph{descendant} of $u$. In the case that there exists $w\in X$ such that $uw=v$, we say that $u$ is a \emph{parent} of $v$ and that $v$ is a \emph{child} of $u$. Notice that this order has a least element, namely the empty word $\epsilon$.

We will denote the group of order-preserving bijections of $X^*$ by $\Aut(X^*)$. Note that this is the group of automorphisms of $X^*$ seen as a poset and should not be confused with the group of automorphisms of the free monoid on $X$. As the Hasse diagram of the poset $(X^*, \leq)$, when rooted at the empty word, is a $|X|$-regular rooted tree, the group $\Aut(X^*)$ can also be seen as the group of automorphisms of a $|X|$-regular rooted tree, and this interpretation is often helpful for visualisation purposes.

Let us now introduce a few special subgroups of automorphisms of $X^*$.

\begin{defn}\label{defn:StabsAndRists}
Let $G\leq \Aut(X^*)$ be a group of automorphisms of $X^*$.
\begin{enumerate}[label=(\roman*)]
\item For $v\in X^*$, the \emph{stabiliser of $v$ in $G$}, denoted by $\St_G(v)$, is the subgroup $\St_G(v)=\{g\in G \mid g\cdot v=v\}$.
\item For $V\subseteq X^*$, the \emph{pointwise stabiliser of $V$ in $G$}, denoted by $\St_G(V)$, is the subgroup $\St_G(V)=\bigcap_{v\in V}\St_G(v)$.
\item For $v\in X^*$, the \emph{rigid stabiliser of $v$ in $G$}, denoted by $\rist_G(v)$, is the subgroup $\rist_G(v)=\{g\in G \mid g\cdot w=w \quad\forall w\in X^* \text{ such that } v\not\leq w\}$.
\item For $V\subseteq X^*$, the \emph{rigid stabiliser of $V$ in $G$}, denoted by $\rist_G(V)$, is the subgroup $\rist_G(V)=\{g\in G \mid g\cdot w = w \quad\forall w\in X^* \text{ not comparable to any } v\in V\}$. It will sometimes be more convenient to look at the intersection $\Srist_G(V) = \rist_G(V) \cap \St_G(V)$ of the rigid stabiliser of a set with the stabiliser of this same set, which we will call the \emph{stabilised rigid stabiliser}.\label{item:RistAndSRist}
\item For $n\in \N$, the \emph{rigid stabiliser of the level $n$ in $G$}, denoted by $\rist_G(n)$, is the subgroup $\rist_G(n)=\langle \bigcup_{v\in X^n}\rist_G(v)\rangle$. Note that this should not be confused with $\rist_G(X^n)$, which is equal to $G$, or with $\Srist_G(X^n)$, which is equal to $\St_G(X^n)$.
\end{enumerate}
\end{defn}

\begin{rem}\label{rem:RistsCommute}
For $n\in \N$, $v\neq w\in X^n$ and $u\in X^*$, notice that if $v\leq u$, then $u$ and $w$ are incomparable. It follows that $\rist_G(v)$ and $\rist_G(w)$ never act non-trivially on the same element of $X^*$, so that these two subgroups commute. Therefore, we have
\[\rist_G(n) \cong \prod_{v\in X^n}\rist_G(v).\]
More generally, if $U,V\subseteq X^*$ are two subsets such that every $u\in U$ is incomparable with every $v\in V$, then the subgroups $\rist_G(U)$ and $\rist_G(V)$ commute.
\end{rem}

It is clear from the definition of the order on $X^*$ that the action of $\Aut(X^*)$ on $X^*$ must preserve the sets $X^n$ for all $n\in \N$. Thus, in terms of transitivity, the most that one can ask of an order-preserving action of a group $G$ on $X^*$ is that it be transitive on each of these sets. This leads to the following definition.

\begin{defn}
Let $G\leq \Aut(X^*)$ be a group of automorphisms of $X^*$. We say that $G$ acts \emph{spherically transitively} on $X^*$ if $G$ acts transitively on $X^n$ for all $n\in \N$.
\end{defn}

We are now ready to define branch groups in the context of groups of automorphisms of $X^*$.

\begin{defn}\label{defn:branch}
Let $G\leq \Aut(X^*)$ be a group of automorphisms of $X^*$. We say that $G$ is a \emph{branch group} if
\begin{enumerate}[label=(\roman*)]
\item the action of $G$ on $X^*$ is spherically transitive,
\item for all $n\in \N$, the subgroup $\rist_G(n)$ is of finite index in $G$.
\end{enumerate}
\end{defn}

\begin{rem}
While we have defined branch groups here in the context of subgroups of $\Aut(X^*)$, one can more generally define branch groups in the group $\Aut(T)$ of automorphisms of a locally finite rooted tree $T$, although we will not need such generality here. We refer the reader to~\cite{BartholdiGrigorchukSunic03} for a more detailed account of branch groups in general.
\end{rem}

\subsection{Self-similar groups and regular branch groups}

It follows from the definition of the order on $X^*$ that for every $g\in \Aut(X^*)$ and every $v\in X^*$, there exists a unique element $g_v\in \Aut(X^*)$ such that
\[g\cdot (vw) = (g\cdot v)(g_v\cdot w)\]
for every $w\in X^*$.
Thus, we can define a map
\begin{align*}
\varphi\colon \Aut(X^*)\times X^* &\rightarrow \Aut(X^*) \\
(g,v)&\mapsto g_v.
\end{align*}
For every $v\in X^*$, we have a map $\varphi_v\colon \Aut(X^*) \rightarrow \Aut(X^*)$ given by $\varphi_v(g)=\varphi(g,v)$.
The map $\varphi$ is a cocycle, meaning that
\[\varphi_v(gh) = \varphi_{h\cdot v}(g)\varphi_v(h)\]
for all $v\in X^*$, $g,h\in \Aut(X^*)$. Furthermore, the map $v \mapsto \varphi_v$ is a semigroup antihomomorphism between the semigroup $X^*$ and the semigroup of self-maps of $\Aut(X^*)$ with respect to composition. In other words, for every $v,w\in X^*$, we have
\[\varphi_{vw} = \varphi_w\circ\varphi_v.\]

The maps $\varphi_v$ allow us to define the notion of self-similar subgroups of $\Aut(X^*)$.

\begin{defn}\label{defn:selfsim}
Let $G\leq \Aut(X^*)$ be a non-trivial group of automorphisms of $X^*$. We say that $G$ is \emph{self-similar} if $\varphi_v(G)\subseteq G$ for all $v\in X^*$. A self-similar group $G$ is said to be \emph{self-replicating} if $\varphi_v(\St_G(v)) = G$ for all $v\in X^*$.
\end{defn}
Self-replicating groups are necessarily infinite, but this is not the case for self-similar groups.

For self-replicating groups, there exists a stronger notion than being a branch group, namely the notion of a regular branch group, which we define below.

\begin{defn}\label{defn:regbranch}
Let $G\leq \Aut(X^*)$ be a self-similar group acting spherically transitively on $X^*$ and let $K\leq G$ be a subgroup of $G$ of finite index. We say that $G$ is \emph{regular branch over $K$} if we have $K\leq \varphi_x(\rist_K(x))$ for all $x\in X$.
\end{defn}

Note that if $G$ is regular branch over $K$, then $G$ is a branch group.
Let us now give examples of self-replicating and regular branch groups.

\begin{example}[The Grigorchuk group]\label{example:GrigorchukGroup}
Let $X=\{0,1\}$. Let $a, b, c, d\in \Aut(X^*)$ be the automorphisms defined by
\begin{align*}
a\cdot 0w &= 1w & a\cdot 1w &= 0w\\
b\cdot 0w &= 0(a\cdot w) & b\cdot 1w &= 1 (c\cdot w)\\
c\cdot 0w &= 0(a\cdot w) & c\cdot 1w &= 1 (d\cdot w)\\
d\cdot 0w &= 0w & d\cdot 1w &= 1 (b\cdot w)
\end{align*}
for all $w\in X^*$, and let $\Grig=\langle a,b,c,d \rangle$ be the subgroup of $\Aut(X^*)$ generated by $a,b,c,d$. This group is known as the \emph{Grigorchuk group} (or \emph{the first Grigorchuk group}). One can check that $\Grig$ is a self-replicating regular branch group over the subgroup $K=\langle [a,b] \rangle_\Grig$, where $\langle [a,b] \rangle_\Grig$ denotes the normal closure in $\Grig$ of the subgroup generated by $[a,b]=aba^{-1}b^{-1}$. A proof of this fact, as well as a general introduction on $\Grig$, can be found for example in~\cite{delaHarpe00}.
\end{example}

\begin{example}[GGS groups]\label{example:GGS}
Let $d\geq2$ be an integer and let $X=\{0,\dots,d-1\}$.
Let $E=(e_0,\dots,e_{d-2})$ be a vector in $(\Z/d\Z)^{d-1}$.
 Let $a, b\in \Aut(X^*)$ be the automorphisms defined by
\begin{align*}
a\cdot iw &= (i+1)w\qquad 0\leq i\leq d-1 \textnormal{, where $i+1$ is taken modulo }d  \\
b\cdot iw &=
\begin{cases}
i(a^{e_i}\cdot w) &\textnormal{ if }0\leq i\leq d-2\\
(d-1) (b\cdot w)&\textnormal{ if } i=d-1
\end{cases}
\end{align*}
for all $w\in X^*$, and let $G=G_E=\langle a,b \rangle$ be the subgroup of $\Aut(X^*)$ generated by $a$ and $b$. These groups are known as the \emph{GGS-groups}, which stands for Grigorchuk-Gupta-Sidki.
If $d$ is prime, then $G$ is a self-replicating group~\cite[Lemma 4.2]{UriaAlbizuri16}, and if furthermore $E$ is not a constant vector, then $G$ is a regular branch group over the subgroup $\gamma_3(G)=\langle[G',G] \rangle$~\cite[Lemma 3.2]{FAZR14}.
Additionally, still under the assumption that $d$ is prime, $G$ is torsion if and only if $\sum_{i=0}^{d-2}e_i=0 \pmod d$~\cite{Vovkivsky98}.
Some of these results have been generalised to the case where $d$ is the power of an odd prime, but as the statements are more complicated in this case, we refer the interested reader to~\cite{DDFAG23} for a full account.
\end{example}

Notice that if $G$ is a regular branch group over $K$, then $\rist_K(1)$ contains a finite index subgroup isomorphic to $K^{|X|}$, where each isomorphic copy of $K$ in this direct product lives in the rigid stabiliser of the corresponding element of $X$. It turns out that if $G$ is self-replicating and regular branch over some $K$, then it must actually be regular branch over a finite index subgroup $L$ such that $\rist_L(1)$ is isomorphic to $L^{|X|}$, as the next proposition shows.

\begin{prop}\label{prop:RegularBranchOverRist}
Let $G\leq \Aut(X^*)$ be a self-replicating regular branch group over a finite index subgroup $K\leq G$. Then, there exist some $N\in \N$ and a finite index normal subgroup $K\leq L\trianglelefteq G$ such that $\varphi_v(\rist_G(v))=L$ for all $v\in X^n$ with $n\geq N$. In particular, $G$ is regular branch over $L$ and $\varphi_x(\rist_L(x))=L$ for all $x\in X$.
\end{prop}
\begin{proof}
Since $G$ is regular branch over $K$, we have $K\leq \varphi_v(\rist_G(v))$ for all $v\in X^*$, and since $K$ is of finite index in $G$, it must also be of finite index in $\varphi_v(\rist_G(v))$. For every $v\in X^*$, the subgroup $\varphi_v(\rist_G(v))$ must be normal in $\varphi_v(\St_G(v))=G$, where this last equality is due to the fact that $G$ is self-replicating. Using this, the fact that $\varphi$ is a cocycle and the fact that for any $w\in X^{|v|}$, the subgroups $\rist_G(v)$ and $\rist_G(w)$ are conjugate, we conclude that $\varphi_v(\rist_G(v))=\varphi_w(\rist_G(w))$ for all $w\in X^{|v|}$. Therefore, it suffices to prove the result for one vertex on each level.

Let us fix some $x\in X$ and some $n\in \N$, and let us consider $\varphi_{x^n}(\rist_G(x^n))$, where $x^n$ denotes the product of $x$ with itself $n$ times in the free semigroup $X^*$. Using the fact that $v\mapsto \varphi_v$ is a antihomomorphism and that
\[\varphi_v(\rist_G(vw)) = \rist_{\varphi_v(\rist_G(v))}(w)\]
for any $v,w\in X^*$, we find
\[K\leq\varphi_{x^n}(\rist_G(x^n)) = \varphi_{x^{n-1}}(\rist_{\varphi_x(\rist_G(x))}(x^{n-1}))\leq \varphi_{x^{n-1}}(\rist_G(x^{n-1})).\]
If we denote by $I_n$ the index of $K$ in $\varphi_{x^{n}}(\rist_G(x^n))$, it follows from the above that $I_n$ is a non-increasing sequence of positive integers, and it must therefore stabilise after some $N\in \N$. 

Let us write $L=\varphi_{x^N}(\rist_G(x^N))$. Since
\[K\leq \varphi_{x^{N+1}}(\rist_G(x^{N+1}))\leq \varphi_{x^N}(\rist_G(x^N)) = L,\]
the fact that $I_{N+1}=I_N$ implies that $\varphi_{x^{N+1}}(\rist_G(x^{N+1}))=L$. By induction, we conclude that $\varphi_{x^{n}}(\rist_G(x^{n}))=L$ for all $n\geq N$. On the other hand,
\[\varphi_{x^{N+1}}(\rist_G(x^{N+1})) = \varphi_{x}(\rist_{\varphi_{x^N}(\rist_G(x^N))}(x))=\varphi_{x}(\rist_{L}(x)),\]
from which we conclude that $L=\varphi_x(\rist_L(x))$,
 which finishes the proof.
\end{proof}

Following~\cite{GarridoSunic23}, let us define the notion of a \emph{maximal branching subgroup}.
\begin{defn}\label{defn:MaxBranchingSubgroup}
Let $G$ be a regular branch group. We say that $K$ is a \emph{maximal branching subgroup} of $G$ if $G$ is regular branch over $K$ and $K$ is maximal for this property.
\end{defn}
 
It follows from the definition that if $G$ is a regular branch over some $K$, then $K$ is contained in some maximal branching subgroup.
 As a direct corollary of Proposition~\ref{prop:RegularBranchOverRist}, we obtain:
\begin{cor}\label{cor:MaximalBranchingSubgroup}
Let $G$ be a self-replicating regular branch group and let $K$ be a maximal branching subgroup.
Then, there exists some $N\in \N$ such that $\varphi_v(\rist_G(v))=K$ for all $v\in X^n$ with $n\geq N$.
In particular, $K$ is unique.
\end{cor}

\section{Normal and almost normal subgroups of branch groups}\label{sec:normal}

We will now establish a few useful results regarding normal and almost normal subgroups of branch groups.
Note that these results are valid for all branch groups, and not only those acting on regular rooted trees.
However, in order to keep our notation as simple as possible, we will state them only for regular branch group, which is the context in which we will need them later on.

\subsection{Normal subgroups of branch groups}

Non-trivial normal subgroups of branch groups always contain the commutator subgroup of the rigid stabiliser of some level, as was shown by the second author.

\begin{prop}[\cite{Grigorchuk00}, see Theorem 4]\label{prop:NormalSubgroupsContainRist'}
Let $G\leq \Aut(X^*)$ be a branch group and let $N\trianglelefteq G$ be a non-trivial normal subgroup. Then, there exists $n\in \N$ such that $\rist_G(n)'\leq N$, where $\rist_G(n)'$ denotes the commutator subgroup of $\rist_G(n)$.
\end{prop}

Using essentially the same proof, one can obtain the following more general result (see for instance~\cite[Lemma 2.11]{Francoeur22p} for a proof).

\begin{prop}\label{prop:R'inSubgroupNormalisedByR}
Let $G\leq \Aut(X^*)$ be a branch group and let $H\leq G$ be a non-trivial subgroup of $G$. Let $v\in X^*$ be an element that is not fixed by $H$, and let $R\leq \rist_{G}(v)$ be a subgroup normalising $H$. Then, $R'\leq H$.
\end{prop}

The proposition above is only interesting, of course, if one knows that the subgroup $R'$ is not trivial, but this is always the case if $R$ is normal.

\begin{prop}\label{prop:R'Infinite}
Let $G\leq \Aut(X^*)$ be a weakly branch group, let $v\in X^*$ be any element and let $R\trianglelefteq \rist_G(v)$ be a non-trivial normal subgroup of $\rist_G(v)$. Then, $R'$ is infinite.
\end{prop}
\begin{proof}
It suffices to show that for every finite subset $F\subseteq R'$, there exists some element $g\in R'$ such that $g\notin F$.
Therefore, let $F\subseteq R'$ be a finite subset of~$R'$ and let $r$ be a non-trivial element of $R$.
There must exist some $n\in \N$ such that $r$ as well as every non-trivial element of $F$ act non-trivially on $vX^n$.

As $r$ acts non-trivially on $vX^n$, there exist $w, w'\in X^n$ with $w'\ne w$ such that $r\cdot vw= vw'$.
Let $s\in \rist_G(vw)$ be any non-trivial element, which exists because $G$ is weakly branch.
Then, there must exist $u, u'\in X^*$ such that $s\cdot vwu= vwu'$ with $u'\ne u$.
We choose a non-trivial element $t\in \rist_G(vwu)$.
Since $R$ is normal in $\rist_G(v)$, we have $[r,s^{-1}]\in R$ and $[r,t^{-1}]\in R$, so that $[[r,s^{-1}],[r,t^{-1}]]\in R'$.

If we can prove that $[[r,s^{-1}],[r,t^{-1}]]$ is non-trivial, then we are done, since $[[r,s^{-1}],[r,t^{-1}]]$ fixes every element of $vX^n$ by construction, whereas all non-trivial elements of $F$ must act non-trivially on $vX^n$.
Thus, it only remains to prove that $[[r,s^{-1}],[r,t^{-1}]]\ne 1$.
For this, let us first notice that
\[\varphi_{vw}([r,s^{-1}]) = \varphi_{vw}(rs^{-1}r^{-1})\varphi_{vw}(s) = \varphi_{vw}(s),\]
since $rs^{-1}r^{-1}\in \rist_G(vw')$ with $w'\ne w$.
Similarly, $\varphi_{vw}([r,t^{-1}]) = \varphi_{vw}(t)$.
Thus, $\varphi_{vw}([[r,s^{-1}],[r,t^{-1}]]) = \varphi_{vw}([s,t])$, so that $\varphi_{vwu}([[r,s^{-1}],[r,t^{-1}]]) = \varphi_{vwu}([s,t])$.
Using similar computations to the ones above, we find that $\varphi_{vwu}([s,t]) = \varphi_{vwu}(t^{-1})$, since $sts^{-1}\in \rist_G(vwu')$ with $u'\ne u$.
We must have $\varphi_{vwu}(t^{-1})\ne 1$, since $t$ is a non-trivial element of $\rist_G(vwu)$ and $\varphi_{vwu}$ is injective when restricted to $\rist_G(vwu)$.
This finishes showing that $[[r,s^{-1}],[r,t^{-1}]]$ is non-trivial and thus concludes the proof.
\end{proof}

As Proposition~\ref{prop:NormalSubgroupsContainRist'} shows, every non-trivial normal subgroup of a branch group contains the derived subgroup of a rigid stabiliser. One might wonder when finite index normal subgroups of a branch group must actually contain the full rigid stabiliser. The groups for which this holds are said to have \emph{trivial branch kernel}.

\begin{defn}\label{defn:TrivialBranchKernel}
A branch group $G\leq \Aut(X^*)$ is said to have \emph{trivial branch kernel} if for every normal subgroup $N\trianglelefteq G$ of finite index, there exists $n\in \N$ such that $\rist_G(n)\leq N$.
\end{defn}

The terminology \emph{trivial branch kernel} comes from the fact that, given a branch group $G\leq \Aut(X^*)$, one can define completions of $G$ with respect to various different directed families of normal subgroups. In particular, one can form the completion $\hat{G}$ with respect to the family of all normal subgroups of finite index, which is called the \emph{profinite completion}, and the completion $\tilde{G}$ of $G$ with respect to the family of rigid stabilisers of levels, which is called the \emph{branch completion}. As rigid stabilisers of levels are normal subgroups of finite index of $G$, we have a natural surjective map $\hat{G}\rightarrow \tilde{G}$, and the kernel of this map is known as the \emph{branch kernel}. Having a trivial branch kernel in the sense of the definition above is equivalent to the branch kernel being trivial.

\begin{rem}\label{rem:CSP}
Although we will not make use of it in the present article, we should also mention that there is another natural completion that one can form with a branch group, namely the \emph{congruence completion} $\overline{G}$, which is formed with respect to the family of pointwise stabilisers of levels of the tree. We have natural surjective homomorphisms $\hat{G}\rightarrow \overline{G}$ and $\tilde{G}\rightarrow \overline{G}$ whose kernels are called respectively the \emph{congruence kernel} and the \emph{rigid kernel}.
Having a trivial congruence kernel is obviously a stronger condition than having a trivial branch kernel, and groups satisfying this are said to possess the \emph{congruence subgroup property}, or CSP for short.
We refer the interested reader to~\cite{BartholdiSiegenthalerZalesskii12} for more details regarding these completions of branch groups.
\end{rem}

\subsection{Almost normal subgroups of branch groups}

In what follows, we will need to have a good understanding of the structure of almost normal subgroups of branch groups. Let us first recall the definition of an almost normal subgroup.

\begin{defn}
Let $G$ be a group and let $H\leq G$ be a subgroup of $G$. We say that $H$ is almost normal in $G$ if the normaliser of $H$ is of finite index in $G$.
\end{defn}

Normal subgroups and subgroups of finite index are obviously almost normal, but in general these are not the only ones.
Let us look at an important example of almost normal subgroups in branch groups.

\begin{example}\label{example:RistAreAlmostNormal}
Let $G\leq \Aut(X^*)$ be a branch group and let $V\subset X^*$ be a finite subset.
Let $n=\max\{|v|\mid v\in V\}$ be the maximum of the length of the vertices in $V$.
Then, the subgroup $\rist_G(V)$ is almost normal in $G$, since its normaliser contains $\rist_G(n)$, which is of finite index in $G$.
\end{example}

As the previous example shows, unlike normal subgroups, almost normal subgroups of branch groups need not always contain the commutator subgroup of the rigid stabiliser of some level. However, it turns out that they always contain the commutator subgroup of one of the subgroups of Example \ref{example:RistAreAlmostNormal}. Before we can make this more precise, we first need to understand the orbits in $X^*$ of finite index subgroups of branch groups.

\begin{prop}\label{prop:FiniteIndexSubgroupsSphericallyTransitiveAction}
Let $G\leq \Aut(X^*)$ be a branch group, and let $H\leq G$ be a finite index subgroup. Then, there exists some $N\in \N$ such that for all $v\in X^N$ and for all $n\in \N$, $\St_H(v)$ acts transitively on $vX^n$.
\end{prop}

We refer the reader to~\cite[Lemma 2.5]{Leemann20} or~\cite[Lemma 2.9]{Francoeur22p} for a proof.

Using this, one can obtain the following structural result for almost normal subgroups of branch groups, which is essentially Theorem 1.2 of~\cite{GarridoWilson14}.

\begin{thm}[cf.~\cite{GarridoWilson14}, Theorem 1.2]\label{thm:StructureAlmostNormalSubgroups}
Let $G\leq \Aut(X^*)$ be a branch group, and let $H\leq G$ be a non-trivial almost normal subgroup of $G$. Then, there exist $n\in \N$ and a non-empty $V\subseteq X^n$ such that
\[\rist_G(n)'\cap \rist_G(V)\leq H \leq \rist_G(V).\]
\end{thm}
\begin{proof}
By \cite{GarridoWilson14}, Theorem 1.2, there exist $n\in \N$ and $V\subseteq X^n$ such that $\prod_{v\in V}\rist_G(v)'\leq H$ and
\[H\cap \Biggl(\prod_{v\in X^n\setminus V}\rist_G(v)\Biggr) = \Biggl[H, \Biggl(\prod_{v\in X^n\setminus V}\rist_G(v)\Biggr)\Biggr] = 1.\]
Note that what is written here is the translation of the result of \cite{GarridoWilson14} in our notation, which is different.
In particular, the reader should be aware of the fact that the subgroup we denote by $\rist_G(V)$ is not the same as the one denoted by $\text{rist}_G(V)$ in \cite{GarridoWilson14}.

Since $\prod_{v\in V}\rist_G(v)' = \rist_G(n)'\cap \rist_G(V)$, we directly have $\rist_G(n)'\cap \rist_G(V)\leq H$.
Now, let $w\in X^*$ be a vertex that is not fixed by $H$.
Without loss of generality, we may assume that $|w|\geq n$, so that there exists a unique $w'\in X^n$ with $w'\leq w$.
For any $h\in H$, we have $h\rist_G(w)h^{-1} = \rist_G(hw)$, and since $w$ is not fixed by $H$, it is clear that $[H, \rist_G(w)]\ne 1$.
It follows, using the fact that $\rist_G(w)\leq \rist_G(w')$, that $w'\notin X^n \setminus V$.
Therefore, $w$ must be comparable to an element of $V$.
Consequently, by definition, $H\leq \rist_G(V)$.
\end{proof}

A corollary of the previous theorem is that the center of an almost normal subgroup of a branch group is always trivial.

\begin{cor}\label{cor:CenterAlmostNormalIsTrivial}
Let $G\leq \Aut(X^*)$ be a branch group and let $H$ be a non-trivial almost normal subgroup of $G$. Then, $H$ is infinite and has trivial center.
\end{cor}
\begin{proof}
We first observe that $H$ is infinite.
Indeed, by Theorem~\ref{thm:StructureAlmostNormalSubgroups}, $H$ contains $\rist_G(w)'$ for some $w$ of level $n$, and by Proposition~\ref{prop:R'Infinite}, $\rist_G(w)'$ is infinite.

We now prove that $H$ has trivial center.
Let $n\in \N$ and $V\subseteq X^n$ be as in Theorem~\ref{thm:StructureAlmostNormalSubgroups}, and let $h\in H$ be a non-trivial element of $H$. Then, there exist some $w\in X^*$ such that $h\cdot w\ne w$, and since $H\leq \rist_G(V)$, this $w$ has to be comparable with an element of $V$. Since $h\cdot w' \ne w'$ for  all $w'\geq w$, we may assume without loss of generality that $|w|\geq n$.

By assumption, $\rist_G(n)'\cap\rist_G(V)\leq H$, which implies that $\rist_G(w)'\leq H$ thanks to our assumptions on $w$. Now, since $h\cdot w\ne w$, we have
\[h\rist_G(w)'h^{-1} = \rist_G(h\cdot w)'\ne \rist_G(w)',\]
which shows that $h$ does not belong to the center of $H$.
\end{proof}

As another consequence of Theorem \ref{thm:StructureAlmostNormalSubgroups}, we obtain that every almost normal subgroup of a finitely generated branch group is finitely generated.

\begin{thm}\label{thm:AlmostNormalFinitelyGenerated}
Let $G\leq \Aut(X^*)$ be a finitely generated branch group.
Then, all almost normal subgroups of $G$ are finitely generated.
\end{thm}
\begin{proof}
Let $H\leq G$ be an almost normal subgroup of $G$.
If $H$ is trivial, then it is obviously finitely generated.
If $H$ is non-trivial, by Theorem \ref{thm:StructureAlmostNormalSubgroups}, there exist $n\in \N$ and a non-empty $V\subseteq X^n$ such that
\[\rist_G(n)'\cap \rist_G(V) \leq H \leq \rist_G(V).\]
It follows that $\rist_G(n)'\cap H = \rist_G(n)'\cap \rist_G(V)$, and since $\rist_G(n)'$ is normal in $G$, the subgroup $\rist_G(n)'\cap H$ is normal in $H$.

Let us consider the quotient $H/(\rist_G(n)'\cap H)$.
We have $H/(\rist_G(n)'\cap H) \cong H\rist_G(n')/ \rist_G(n)'\leq G/\rist_G(n)'$.
Notice that $G/\rist_G(n)'$ is a finitely generated virtually abelian group, since $G$ is a finitely generated branch group.
It is a standard fact and an easy exercise that subgroups of finitely generated virtually abelian groups are themselves finitely generated (that is, finitely generated virtually abelian groups are Noetherian).
We therefore conclude that $H/(\rist_G(n)'\cap H)$ is finitely generated.
Therefore, to show that $H$ is finitely generated, it suffices to show that $\rist_G(n)'\cap H$ is.

By \cite[Corollary A.5]{CapraceLeBoudec23}, since $G$ is a finitely generated branch group, $\rist_G(n)'$ is finitely generated, and since $\rist_G(n)' = \prod_{v\in X^n}\rist_G(v)'$, it follows that $\rist_G(v)'$ is finitely generated for all $v\in X^n$.
Consequently
\[\rist_G(n)'\cap H = \rist_G(n)'\cap \rist_G(V) = \prod_{v\in V} \rist_G(v)'\]
is finitely generated, which concludes the proof.
\end{proof}

We now define the notion of a complement of an almost normal subgroup, which will be particularly useful in the context of branch groups, as we will see below.

\begin{defn}
Let $G$ be a group and let $H\leq G$ be an almost normal subgroup of $G$. An almost normal subgroup $K\leq G$ is called a \emph{complement} of $H$ if $[H,K]=1$ and $H\cap K=1$ (or in other words, if $HK$ is the direct product of $H$ and $K$) and if $HK$ is of finite index in $G$.
\end{defn}

In just infinite branch groups, almost normal subgroups always admit a complement.

\begin{prop}
Let $G\leq \Aut(X^*)$ be a just infinite branch group, and let $H\leq G$ be an almost normal subgroup of $G$. Then, there exists an almost normal subgroup $K\leq G$ such that $K$ is a complement for $H$.
\end{prop}
\begin{proof}
If $H=1$, we can take $K=G$.
Otherwise, by Theorem~\ref{thm:StructureAlmostNormalSubgroups}, there exist $n\in \N$ and $V\subseteq X^n$ such that $\rist_G(n)'\cap \rist_G(V)\leq H\leq \rist_G(V)$.
Let $W=X^n\setminus V$ and let $K=\rist_G(n)'\cap \rist_G(W)$. Clearly, $K$ is an almost normal subgroup of $G$.

Since $V, W\subseteq X^n$ with $V\cap W=\emptyset$, it is easy to see from the definition that we must have $\rist_G(V)\cap \rist_G(W)=1$ and $[\rist_G(V),\rist_G(W)]=1$. Thus, \emph{a fortiori}, $H\cap K=[H,K]=1$. Furthermore, since $V\cup W=X^n$, and since $\rist_G(n)'\cap \rist_G(V) = \prod_{v\in V}\rist_G(v)'$, $\rist_G(n)'\cap \rist_G(W) = \prod_{v\in W}\rist_G(v)'$, we conclude that
\[\rist_G(n)' = \left(\rist_G(n)'\cap \rist_G(V)\right)\left(\rist_G(n)'\cap \rist_G(W)\right)\leq HK.\]
Using the fact that $\rist_G(n)'$ is normal in $G$ and that $G$ is just infinite, we conclude that $HK$ is of finite index in $G$, which shows that $K$ is a complement for $H$.
\end{proof}

In branch groups, commuting almost normal subgroups must always lie (up to finite index) in the complement of each other, as we will now show.

\begin{prop}\label{prop:CommutingAlmostNormalLiesInComplement}
Let $G\leq \Aut(X^*)$ be a just infinite branch group, and let $H_1,H_2\leq G$ be two commuting almost normal subgroups of $G$. Let $K_1,K_2\leq G$ be complements of $H_1$ and $H_2$, respectively. Then, $H_1\cap K_2$ is of finite index in $H_1$ (and similarly, $H_2\cap K_1$ is of finite index in $H_2$).
\end{prop}
\begin{proof}
By assumption, we have $H_2K_2$ of finite index in $G$. Therefore, $H_1\cap H_2K_2$ is of finite index in $H_1$. It thus suffices to show that $H_1\cap H_2K_2= H_1\cap K_2$. Let $g\in H_1\cap H_2K_2$ be any element. Then, there exist $h\in H_2$, $k\in K_2$ such that $g=hk$. Since $H_1$ and $K_2$ both commute with $H_2$ by assumption, it follows that $g$ and $k$ both commute with $H_2$, which forces $h$ to be in the center of $H_2$. By Corollary~\ref{cor:CenterAlmostNormalIsTrivial}, this means that $h=1$, which shows that $g\in H_1\cap K_2$ and thus that $H_1\cap H_2K_2 = H_1\cap K_2$, since $g$ was arbitrary.
\end{proof}

\section{Tree-primitive actions}\label{section:TreePrimitivity}

In order to prove our main result, we will need to make an extra assumption about the action of our group on its corresponding rooted tree, which we call \emph{tree-primitivity}. It can be seen as an adaptation of primitivity to groups acting on rooted trees.
Recall that an action of a group $G$ on a set is called \emph{primitive} if there are no non-trivial $G$-invariant partitions of the set. For any $n>1$, the action of a group $G\leq \Aut(X^*)$ on $X^n$ can never be primitive, since $G$ preserves the non-trivial partitions $\{vX^{n-m}\}_{v\in X^m}$ for all $0<m<n$. Our notion of tree-primitivity is thus a version of primitivity adjusted to take this fact into account.

\begin{defn}\label{defn:treeprim}
Let $G\leq \Aut(X^*)$ be a group of automorphisms of $X^*$. We say that the action of $G$ on $X^*$ is \emph{tree-primitive} if for every $n\in \N$, the only $G$-invariant partitions of $X^n$ are the partitions $\{vX^{n-m}\}_{v\in X^m}$ for $0\leq m \leq n$.
\end{defn}

\begin{rem}
For any group $G\leq \Aut(X^*)$, its action on $X^*$ naturally extends to an action on $X^{\infty}$, the set of right-infinite words in the alphabet $X$. By declaring all sets of the form $vX^{\infty}$ open, for all $v\in X^*$, one turns $X^{\infty}$ into a topological space, homeomorphic to the Cantor set. One could define an analogue (and possibly more natural) version of tree-primitivity for the action of $G$ on $X^\infty$ by declaring that the action of $G$ on $X^\infty$ is \emph{tree-primitive} if the only $G$-invariant partitions of $X^\infty$ into clopen sets are the partitions $\{vX^{\infty}\}_{v\in X^n}$ for $n\in \N$. It is straightforward to show that the action of $G$ on $X^*$ is tree-primitive if and only if the action of $G$ on $X^\infty$ is tree-primitive.
\end{rem}

\begin{rem}
Let $G\leq \Aut(X^*)$ be a branch group with a tree-primitive action on $X^*$. Then, using Rubin's theorem~\cite[Theorem 3.1]{Rubin96}, one can show that the branch action of $G$ is unique, in the sense of Grigorchuk and Wilson~\cite{GrigorchukWilson03b}.
\end{rem}

We now present sufficient conditions for the action of a group to be tree-primitive. In light of the above remark, it is perhaps not surprising that these conditions appear to be closely related to the necessary conditions for the uniqueness of the branch action obtained by the second author and Wilson in~\cite{GrigorchukWilson03b}.

\begin{prop}\label{prop:TreePrimitivity}
Let $G\leq \Aut(X^*)$ be a group acting spherically transitively on~$X^*$. Suppose that
\begin{enumerate}[label=(\roman*)]
\item for every $v\in X^*$, the action of $\St_G(v)$ on $vX$ is primitive,\label{item:PrimitiveAction}
\item for every $n\geq 1$, for every pair of distinct elements $u, v\in X^{n}$ and for every $x,y\in X$, there exists $g\in G$ such that $g\cdot vx = vx$ but $g\cdot uy = uy'$ with $y'\ne y$.\label{item:CanFixOneCousinMoveOther}
\end{enumerate}
Then the action of $G$ on $X^*$ is tree-primitive.
\end{prop}
\begin{proof}
Let us fix $n\in \N$, and let $P=\{P_i\}_{i=1}^{k}$ be a $G$-invariant partition of $X^n$. We need to show that this partition must be of the form $P = \{vX^{n-m}\}_{v\in X^{m}}$ for some $m\leq n$. If $n=0$, this is obvious, and if $n=1$, it follows directly from the primitivity of the action of $G$ on $X$. Let us thus assume that $n>1$.

Let $v\in X^{n-1}$ be any element. Then, the collection of non-empty sets of the form $P_i\cap vX$ define a partition of $vX$, which must be $\St_G(v)$-invariant, given that $P$ is $G$-invariant. By the primitivity of the action of $\St_G(v)$ on $vX$, this partition must be trivial. This means that exactly one of the following two cases is true:
\begin{enumerate}[label=(\alph*)]
\item there exists some $i\in \{1,2,\dots, k\}$ such that $P_i\cap vX = vX$,\label{item:CoarsetPartition}
\item $|P_i\cap vX|\leq 1$ for all $1\leq i \leq k$.\label{item:FinestPartition}
\end{enumerate}
Since the action of $G$ on $X^{n-1}$ is transitive, if case~\ref{item:CoarsetPartition} holds for some $v\in X^{n-1}$, then it holds for all $v\in X^{n-1}$ (and likewise, if~\ref{item:FinestPartition} holds for some $v\in X^{n-1}$, then it must hold for all).

Let us first suppose that case~\ref{item:FinestPartition} holds for all $v\in X^{n-1}$. Let us suppose that there exists $1\leq i \leq k$ such that $|P_i|\geq 2$, and let $w_1,w_2\in P_i$ be two different elements. There exist $x_1,x_2\in X$ and $v_1, v_2\in X^{n-1}$ such that $w_1=v_1x_1$ and $w_2=v_2x_2$. Since $|P_i\cap vX|\leq 1$ for all $v\in X^{n-1}$, we must have $v_1\ne v_2$. Therefore, by condition~\ref{item:CanFixOneCousinMoveOther}, there exists $g\in G$ such that $g\cdot w_1=w_1$ but $g\cdot w_2 = v_2 y\ne w_2$ for some $y\in X\setminus \{x_2\}$. Now, again using the fact that $|P_i\cap v_2X|\leq 1$, we see that $g\cdot w_2\notin P_i$, which means that $g\cdot P_i\ne P_i$. However, we have $g\cdot w_1=w_1\in P_i$, which by the $G$-invariance of the partition $P$ implies that $g\cdot P_i=P_i$, a contradiction. We conclude that every set in the partition $P$ must have size one, and thus $P=\{v\}_{v\in X^n}$ has the required form.

Let us now suppose that case~\ref{item:CoarsetPartition} holds for all $v\in X^{n-1}$. In this case, the partition $P$ has the form $P=\{P_i'X\}_{i=1}^{k}$, where $P'=\{P_i'\}_{i=1}^{k}$ is a partition of $X^{n-1}$ which must necessarily be $G$-invariant, as $P$ is. Therefore, by induction on $n$, we can conclude that $P' = \{vX^{n-1-m}\}_{v\in X^{m}}$ for some $m\leq n-1$. It follows that $P = \{vX^{n-m}\}_{v\in X^{m}}$.
\end{proof}

If the group $G$ under consideration is self-replicating, there exists a simpler sufficient condition for tree-primitivity that only asks to understand the action of $G$ on the first two levels of the tree.

\begin{prop}\label{prop:TreePrimitivityReplicating}
Let G be a self-replicating group acting spherically transitively on $X^*$ such that the action of $G$ on $X$ is primitive and such that $\St_G(X)$ acts spherically transitively on $xX^*$ for all $x\in X$. Suppose furthermore that for every $v,w\in X^2$, the subgroup $\St_G(\{v,w\})$ acts spherically transitively on $vX^*$. If there exist $x_0,y_0\in X$ with $x_0\ne y_0$ such that for every $v\in x_0X$ and $w\in y_0X$, we have $\St_G(v)\cap \St_G(y_0) \not\leq \St_G(w)$, then the action of G on $X^*$ is tree-primitive.
\end{prop}
\begin{proof}
We prove the result by induction on the level $n$ on which we consider the $G$-invariant partition.
If $n=1$, then this is just the assumption that the action is primitive on $X$.
Now, suppose that the result holds for any $G$-invariant partition of $X^n$, and let us consider a $G$-invariant partition $\PP$ of $X^{n+1}$.

For every $x\in X$, the restriction of this partition to $xX^{n}$ gives us a $\St_G(x)$-invariant partition of $xX^n$.
Applying $\varphi_x$ and using the fact that $G$ is self-replicating, we obtain a $G$-invariant partition of $X^n$.
Thus, by our induction hypothesis, there must exist some $m_x\leq n$ such that this partition is $\{vX^{n-m_x}\}_{v\in X^{m_x}}$, which means that the restriction of $\PP$ to $xX^n$ is the partition $\{xvX^{n-m_x}\}_{v\in X^{m_x}}$.
Since the action of $G$ on $X$ is transitive, and since $\PP$ is $G$-invariant, it follows that for each pair $x,y\in X$, there exists some $g\in G$ such that the restriction of $\PP$ to $yX^n$ is the image under $g$ of the restriction of $\PP$ to $xX^n$.
In other words, we have $\{yvX^{n-m_y}\}_{v\in X^{m_y}} = \{g\cdot xvX^{n-m_x}\}_{v\in X^{m_x}}$, from which it follows that $m_x=m_y$.
Since $m_x$ is independent of $x$, let us rename it simply $m$.

It follows from the above that the parts of the partition $\PP$ are unions of sets of the form $vX^{n-m}$ for $v\in X^{m+1}$.
If $m<n$, then we can form a partition $\PP'$ of $X^n$ by simply deleting the last letter from the parts of the partition $\PP$.
In other words, two elements $w_1,w_2\in X^n$ belong to the same part of $\PP'$ if and only if for one (and thus for every) $x\in X$, the elements $w_1x$ and $w_2x$ belong to the same part of $\PP$.
It is clear that we can reconstruct the partition $\PP$ from the partition $\PP'$ by multiplying every part of the partition $\PP'$ by $X$ on the right.
Since $\PP$ is $G$-invariant, it is not hard to see that the partition $\PP'$ must also be $G$-invariant.
Thus by our induction hypothesis, the partition $\PP'$ is of the form $\{vX^{n-k}\}_{v\in X^k}$ for some $1\leq k \leq n$, from which it follows that $\PP = \{vX^{n+1-k}\}_{v\in X^k}$, which is what we wanted to show.

It thus only remains to treat the case where $m=n$.
In this case, for any $x\in X$, the intersection of a part of $\PP$ with $xX^n$ contains at most one element.
Let us define a relation $\sim$ on $X$ by $x\sim y$ if and only if there exist $w_1,w_2\in X^n$ such that $xw_1$ and $yw_2$ belong to the same part of the partition $\PP$.
It is clear that $\sim$ is a reflexive and symmetric relation, and it is also obvious from the definition that it is $G$-invariant.
Let $x,y,z\in X$ be such that $x\sim y$ and $y\sim z$.
Then, there exist $w_1, w_2, w_2', w_3\in X^n$ such that $xw_1$ and $yw_2$ belong to the same part, and $yw_2'$ and $zw_3$ belong to the same part.
As the action of $\St_G(X)$ on $yX^{n}$ is transitive, there exists $g\in \St_G(X)$ such that $g\cdot yw_2 = yw_2'$.
Since $g\in \St_G(X)$, there exists $w_1'\in X^n$ such that $g\cdot xw_1 = xw_1'$, and by $G$-invariance of $\PP$, we conclude that $xw_1'$ is in the same part as $yw_2'$, which is itself in the same part as $zw_3$.
This shows that $x\sim z$ and thus that the relation $\sim$ is a $G$-invariant equivalence relation.
Therefore, the equivalence classes of $\sim$ form a $G$ invariant partition of $X$, which by the primitivity of the action of $G$ on $X$ must consist either of singletons or of only one part, namely~$X$.
In the first case, we conclude that the partition $\PP$ consists only of singletons, since no two elements in the same part can have different first letters and no two elements with the same first letter can be in the same part.
The only thing that remains to finish the proof is thus to prove that the second case is impossible.

Let us suppose that $x\sim y$ for all $x,y\in X$, and let $x_0,y_0\in X$ be as in the statement of the proposition, meaning that for every $v\in x_0X$ and $w\in y_0X$, we have $\St_G(v)\cap \St_G(y_0) \not\leq \St_G(w)$.
Since $x_0\sim y_0$, there must exist $x_1,y_1\in X$ and $w_1,w_2\in X^{n-1}$ such that $x_0x_1w_1$ and $y_0y_1w_2$ belong to the same part of the partition.
By our assumption on $x_0$ and $y_0$, there exist $g\in G$ and $y_2\in X$ with $y_2\ne y_1$ such that $g\cdot x_0x_1 = x_0x_1$ but $g\cdot y_0y_1 = y_0y_2$.
Let $w_1' \in X^{n-1}$ be such that $g\cdot x_0x_1w_1 = x_0x_1w_1'$.
By our assumptions, the action of $\St_G(x_0x_1)\cap \St_G(y_0y_2)$ on $x_0x_1X^*$ is spherically transitive.
Therefore, there exists $h\in \St_G(x_0x_1)\cap \St_G(y_0y_2)$ such that $h\cdot x_0x_1w_1' = x_0x_1 w_1$.
Since $h\in \St_G(y_0y_2)$, there exists $w_2'\in X^{n-1}$ such that $h\cdot y_0y_2w_2 = y_0y_2w_2'$.
As the partition $\PP$ is $G$-invariant, the elements $hg\cdot x_0x_1w_1 = x_0x_1w_1$ and $hg \cdot y_0y_1w_2 = y_0y_2w_2'$ belong to the same part of the partition, from which it follows that $y_0y_1w_2$ and $y_0y_2w_2'$ belong to the same part of the partition.
This is impossible, since $y_0y_1w_2\ne y_0y_2w_2'$, but different elements in the same part of the partition $\PP$ must have different first letter.
This case can therefore never occur, and the result is thus proved.
\end{proof}

Using Proposition~\ref{prop:TreePrimitivityReplicating}, we can easily show that both the Grigorchuk group and the GGS groups act tree-primitively.
­\begin{cor}\label{cor:GrigTreePrimitivity}
The action of $\Grig=\langle a,b,c,d \rangle$ (see Example~\ref{example:GrigorchukGroup}) on $\{0,1\}^*$ is tree-primitive.
\end{cor}
\begin{proof}
Let $X=\{0,1\}$.
As discussed in Example~\ref{example:GrigorchukGroup}, $\Grig$ is a self-replicating branch subgroup of $\Aut(X^*)$, and hence acts spherically transitively on $X^*$.
Moreover, since $X=\{0,1\}$, for any $x\in X$ we have $\St_\Grig(x)=\St_\Grig(X)$, which by self-replicatingness acts spherically transitively on $xX^*$.
By~\cite{UriaAlbizuri16}[Proposition 5.3], we have $\varphi_v(\St_\Grig(X^2)) = \Grig$ for all $v\in X^2$, so that the action of $\St_\Grig(v)\cap\St_\Grig(w)$ is spherically transitive on $vX^*$ for all $v,w\in X^2$.
Finally, let $x_0=1$ and $y_0=0$. The element $b=(a,c)$ is in $\St_\Grig(1x)\cap\St_\Grig(0)$ but $b(0y)\neq 0y$ for all $x,y\in X$.
The action of $\Grig$ on $X^*$ hence satisfies the hypothesis of Proposition~\ref{prop:TreePrimitivityReplicating} and is therefore tree-primitive.
\end{proof}

\begin{cor}\label{cor:GGSTreePrimitivity}
The action of a periodic GGS-group $G=G_E$ (see Example~\ref{example:GGS}) on $\{0,\dots,p-1\}^*$ for some prime $p$ is tree-primitive.
\end{cor}
\begin{proof}
Let $X=\{0,\dots,p-1\}$.
As discussed in Example~\ref{example:GGS}, GGS groups are self-replicating subgroups of $\Aut(X^*)$.
By definition, $G$ acts on the first level of~$X$ as the subgroup generated by the cyclic permutation $(0\, 1\,\dots\, p-1)$, which is primitive and transitive.
Together with self-replicatingness, this implies that the action of $G$ on $X^*$ is spherically-transitive.
Moreover, by~\cite{UriaAlbizuri16}[Proposition 5.1], $G$ satisfies\footnote{Not all self-replicating groups satisfy this property, see~\cite{UriaAlbizuri16} for counter-examples} $\varphi_v(\St_G(X^k))=G$ for all $k\in \N$ and all $v\in X^k$, the action of $\St_G(X^k)$ on $vX^*$ is hence spherically transitive for all $v\in X^k$.
Now, let $0\leq i \leq p-2$ be such that $e_i\ne 0$, where we write $E=(e_0,\dots, e_{p-2})$ and let $x_0=p-1$, $y_0=i$. Notice that for every $x,y\in X$, the element $b=(a^{e_0},\dots,a^{e_{p-2}},b)$ is in $\St_G(x_0x)\cap\St_G(y_0)$ but $b(y_0y)\neq y_0y$.
The action of $G$ on $X^*$ hence satisfies the hypotheses of Proposition~\ref{prop:TreePrimitivityReplicating} and is therefore tree-primitive.
\end{proof}

\section{Block subgroups and the subgroup induction property}\label{section:BlockSubgroupsAndSIP}

Let us now introduce the two notions which are central to the present article, namely \emph{block subgroups} and the \emph{subgroup induction property}.

\subsection{Block subgroups}

Block subgroups are subgroups of automorphisms of $X^*$ that possess a special form. Abstractly, block subgroups are simply direct products, but some of the direct factors may be embedded ``diagonally''. To make this precise, let us first introduce the notion of a diagonal block subgroup.

\begin{defn}\label{defn:diag}
Let $G\leq \Aut(X^*)$ be a group of automorphisms of $X^*$ and let $V\subseteq X^*$ be a finite set of pairwise incomparable elements of $X^*$. A subgroup $H\leq G$ is said to be a \emph{diagonal block subgroup with supporting vertex set $V$} if
\begin{enumerate}[label=(\roman*)]
\item $H \leq \Srist_G(V)$, where $\Srist_G(V)$ denotes the stabilised rigid stabiliser of $V$, see Definition~\ref{defn:StabsAndRists}~\ref{item:RistAndSRist},
\item $\varphi_v(H)$ is of finite index in $\varphi_v(\St_G(v))$ for all $v\in V$,
\item for all $v\in V$, $\varphi_v$ is injective on $H$.
\end{enumerate}
\end{defn}

In other words, a diagonal block subgroup with supporting vertex set $V$ is a subgroup $H$ that is abstractly isomorphic to a finite index subgroup of $\varphi_v(\St_G(v))$ for some $v\in V$, but which is embedded in $\Srist_G(V)$ ``diagonally'', in the sense that for every $v,w\in V$ and every $h\in H$, the ``$w$ component of $h$'', $\varphi_w(h)$, is uniquely determined by its ``$v$ component'' $\varphi_v(h)$. To illustrate, let us see a couple of important examples of diagonal block subgroups.

\begin{example}
Let $G\leq \Aut(X^*)$ be a branch group and let $v\in X^*$ be any element. Then, the diagonal block subgroups with supporting vertex set $V=\{v\}$ are exactly the finite index subgroups of $\rist_G(v)$.
\end{example}

\begin{example}
Let $G\leq \Aut(X^*)$ be a self-replicating regular branch group over a subgroup $K$. Then, for every $v\in X^*$, there exists a subgroup $K@v\leq \rist_K(v)$ such that $K@v \cong K$. Let $f_v\colon K \rightarrow K@v$ be one such isomorphism, and let $v,w\in X^*$ be two incomparable elements. Then, the subgroup $H=\{f_v(k)f_w(k) \mid k\in K\}$ is a diagonal block subgroup with supporting vertex set $V=\{v,w\}$. This example can be generalised to bigger sets of incomparable vertices in an obvious way, see Figure~\ref{fig:diag} for an example with supporting vertex set of cardinality $3$.
\end{example}

\begin{figure}[htbp]
\includegraphics{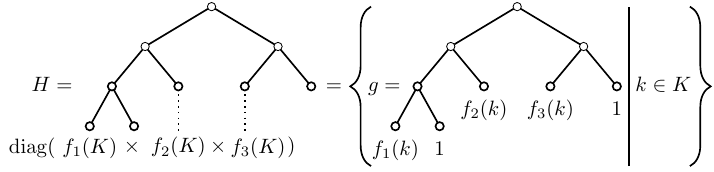}
\caption{A diagonal block subgroup over $K$ with  supporting vertex set of cardinality $3$.}
\label{fig:diag}
\end{figure}

Block subgroups are then simply direct products of diagonal block subgroups over pairwise incomparable supporting vertex sets.

\begin{defn}\label{defn:Block}
Let $G \leq \Aut(X^*)$ be a group of automorphisms of $X^*$ and let $H\leq G$ be a subgroup. Let $V=\bigsqcup_{i=1}^{b}V_i\subseteq X^*$ be a finite set of pairwise incomparable elements of $X^*$ partitioned into $b$ non-empty subsets, and let us denote by $P=\{V_i\}_{i=1}^{b}$ this partition. We say that $H$ is a \emph{block subgroup with supporting partition $P$} if
\[H=H_1\cdots H_b\]
where each $H_i$ is a diagonal block subgroup with supporting vertex set $V_i$. See Figure~\ref{fig:block} for an example.
\end{defn}

\begin{figure}[htbp]
\includegraphics{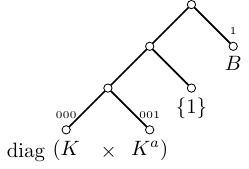}
\caption{A block subgroup with supporting partition $\{\{000,001\},1\}$. Here $H=H_1\cdot H_2$ with $H_1$ a block subgroup over $K$ with supporting partition $\{000,001\}$ and $H_2$ a block subgroup over $B$ with supporting partition $\{B\}$.}
\label{fig:block}
\end{figure}

\begin{rem}
Notice that in the previous definition, each $H_i$ is a subgroup of $\Srist_G(V_i)$. By assumption, when $i\ne j$, any element of $V_i$ is incomparable with any element of $V_j$. Therefore, the subgroups $H_i$ pairwise commute, by Remark~\ref{rem:RistsCommute}. Thus, a block subgroup is truly a group, and it is isomorphic to a direct product of diagonal block subgroups.
\end{rem}

In the case of regular branch groups, one can also define the stronger notion of \emph{regular} block subgroups.

\begin{defn}\label{defn:RegularBlock}
Let $G\leq \Aut(X^*)$ be a regular branch group over a subgroup $K$, and let $H\leq G$ be a block subgroup with supporting partition $P=\{V_i\}_{i=1}^{d}$. We say that $H$ is a \emph{regular block subgroup over $K$} if for all $v\in V=\bigsqcup_{i=1}^{b}V_i$, we have $\varphi_v(H)=K$.
\end{defn}

\subsection{The subgroup induction property}

We now introduce the main property used throughout the article, which is known as the \emph{subgroup induction property}. For this, we first need to define an inductive collection of subgroups of a group acting by automorphisms on $X^*$.
\begin{defn}\label{defn:InductiveClass}
Let $G\leq X^*$ be a self-replicating group of automorphisms of $X^*$ and let $\C$ be a collection of subgroups of $G$.
We say that $\C$ is an \emph{inductive collection} if it satisfies the following properties:
\begin{enumerate}[label=(\Roman*)]
\item $1, G\in \C$,
\item if $H\in \C$ and $L$ contains $H$ as a finite index subgroup, then $L\in \C$,\label{item:ConditionInductiveFiniteIndexOvergroup}
\item if $H\leq \St_G(X)$ is finitely generated and $\varphi_x(H)\in \C$ for all $x\in X$, then $H\in \C$.
\end{enumerate}
\end{defn}

It is clear that if we take $\C$ to be the collection of all subgroups of $G$, then $\C$ must be an inductive collection. Groups with the subgroup induction property are groups for which inductive collections of subgroups cannot be too small, in the sense that they must contain at least all finitely generated subgroups.

\begin{defn}\label{defn:SIP}
Let $G\leq X^*$ be a finitely generated self-replicating group of automorphisms of $X^*$. We say that $G$ has the \emph{subgroup induction property}, or SIP for short, if every inductive collection $\C$ of subgroups of $G$ contains all finitely generated subgroups of $G$.
\end{defn}

\begin{rem}\label{rem:WeakStrong}
In Definition~\ref{defn:InductiveClass}, if we were to replace condition~\ref{item:ConditionInductiveFiniteIndexOvergroup} with

\begin{enumerate}
\item[(II')] if $H\leq L$ is of finite index, then $H\in \C$ if and only if $L\in \C$
\end{enumerate}
this would give rise to a more restrictive definition of an inductive collection, and consequently to a weaker version of the subgroup induction property, which we will call the \emph{weak subgroup induction property}. The reader should be aware that what is called the subgroup induction property in~\cite{GrigorchukLeemannNagnibeda21} corresponds in fact to the weak subgroup induction property under our terminology. However, there is currently no known example of a group with the weak subgroup induction property that does not also possess the stronger version, and it is not known if such a group can exist.
The interested reader can also have a look at~\cite{FrancoeurLeemann20} for interesting consequences of the weak subgroup induction property.

We will show in Theorem~\ref{thm:BlockSubgroups} that, for groups acting tree-primitively, the subgroup induction property implies that every finitely generated subgroup of $G$ is virtually a block subgroup. This last property implies in turn the weak subgroup induction property~\cite[Proposition 4.3]{GrigorchukLeemannNagnibeda21}.
(The article~\cite{GrigorchukLeemannNagnibeda21} claimed that the last two properties were equivalent, but only providing a sketch of a proof.)
\end{rem}

In view of the above, it is natural to ask the following:
\begin{qu}
Does there exist a group that possesses the weak subgroup induction property but not the subgroup induction property?
\end{qu}

There are only a few known examples of groups with the SIP. As far as the authors are aware, at the time of writing, the complete list of known examples consists of the  Grigorchuk group~\cite{GrigorchukWilson03} and the torsion GGS groups~\cite{Garrido16, FrancoeurLeemann20}.

Another natural generalisation of the subgroup induction property is to look at a self-replicating family of groups $(G_i)_i$ instead of a unique self-replicating group $G$. It is possible to modify Definition~\ref{defn:InductiveClass} to define an inductive family $(\C_v)_{v\in X^*}$ of classes of subgroups and to prove a result analogous to the forthcoming Proposition~\ref{prop:ExistNiceTransversal} in this generality. We hence ask
\begin{qu}
To what extent can results about self-replicating branch groups with the subgroup induction property be extended to self-replicating families of branch groups with the subgroup induction property?
\end{qu}

Let us now introduce a useful bit of terminology that will make it easier to state certain results about groups with the subgroup induction property.

\begin{defn}
Let $T\subseteq X^*$ be a finite set. We will say that $T$ is a \emph{transversal} of $X^*$ if elements of $T$ are pairwise incomparable, but for all $v\in X^*\setminus T$, there exists $t\in T$ such that $v$ and $t$ are comparable.
\end{defn}
\begin{rem}
If $T$ is a transversal of $X^*$, then $T$ must be finite.
\end{rem}
\begin{example}
For any $n\in \N$, the set $X^n$ is a transversal of $X^*$.
\end{example}

It was shown in \cite[Proposition 4.3]{GrigorchukLeemannNagnibeda21} that if $G$ has the weak subgroup induction property, then finitely generated subgroups of $G$ have ``nice sections''. A similar statement holds for the subgroup induction property.

\begin{prop}\label{prop:ExistNiceTransversal}
Let $G$ be a finitely generated self-replicating group of automorphisms of $X^*$ possessing the subgroup induction property, and let $H\leq G$ be a finitely generated subgroup of $G$. Then, there exists a transversal $T$ of $X^*$ such that for every $v\in T$, either $\varphi_v(\St_H(v))$ is finite or $\varphi_v(\St_H(v))=G$.
\end{prop}
\begin{proof}
Let $\C$ be the collection of all finitely generated subgroups $H$ of $G$ such that there exists a transversal $T$ of $X^*$ such that for every $v\in T$, either $\varphi_v(\St_H(v))$ is finite or $\varphi_v(\St_H(v))=G$.
We will show that $\C$ is an inductive class, and hence contains all finitely generated subgroups of $G$ by the subgroup induction property.

Both $\{1\}$ and $G$ belong to $\C$, for the transversal $T$ consisting of the root of $X^*$.

If $H\in\C$ for some transversal $T$ and $L$ contains $H$ as a finite index subgroup, then $L$ is also in $\C$ for the same transversal $T$.

Finally, let $H\leq\St_G(X)$ be finitely generated with $\varphi_x(H)\in\C$ for all $x\in X$.
Then for every $x\in X$ we have a transversal $T_x$ of the subtree $\{w\in X^* \mid w\geq x\}$ of vertices below $x$ witnessing that $\varphi_x(H)$ is in $\C$.
The (disjoint) union of all of the $T_x$ forms a transversal $T$ of $X^*$.
%Now, every $v\in T$ belongs to some $T_x$ and, since $H$ stabilizes $X$, we have $\varphi_v(\St_H(v))=\varphi_v(\St_{\varphi_x(H)}(v))$ is either finite or equal to~$G$.
Now, every $v\in T$ belongs to some $T_x$ and hence we can write $v=xw$. Since $H$ stabilizes $X$, we have $\varphi_v(\St_H(v))=\varphi_w(\St_{\varphi_x(H)}(w))$, which is hence either finite or equal to~$G$.
\end{proof}

Observe that the converse of Proposition \ref{prop:ExistNiceTransversal} holds in any finitely generated branch group. Indeed, it follows from the following result.
\begin{lemma}\label{lem:NiceTransversalFG}
Let $H$ be a subgroup of a finitely generated branch group $G\leq\Aut(X^*)$. Suppose that there exists a transversal $T$ such that for every $v\in T$, the section $\varphi_v(\St_H(v))$ is either finite or of finite index in $\varphi_v(\St_G(v))$. Then $H$ is finitely generated.
\end{lemma}
\begin{proof}
Let $T=T_i\sqcup T_f$, where $T_i=\{v\in T \mid [\varphi_v(\St_G(v)):\varphi_v(\St_H(v))]<\infty\}$ and $T_f = \{v\in T \mid |\varphi_v(\St_H(v))|<\infty\}$.
Let $T_i=\{v_1,\dots, v_d\}$ and let $K=\Srist_H(T_i)$.
It follows from the definition of $T_i$ and $T_f$ that $K$ is of finite index in $H$.
Therefore, it is finitely generated if and only if $H$ is.
Moreover, if $v\in T_i$, then $\varphi_v(K)$ is of finite index in $\varphi_v(\St_G(v))$ and therefore finitely generated (and of course, if $v\in T_f$, then $\varphi_v(K)=1$).
Let $X_1$ be a finite generating set of %$\varphi_{v_1}(\St_K(v_1))$
$\varphi_{v_1}(K)$ and let $\tilde X_1$ be a subset of $K$ containing exactly one preimage by $\varphi_{v_1}$ of every element of $X_1$.
Then $K$ is generated by $\tilde X_1$ together with $K_1=\rist_K(T_i\setminus v_1)$.
Since the subgroup $K_1$ is normal in $K$, the section $\varphi_{v_2}(K_1)$ is an almost normal subgroup of %$\varphi_{v_2}(G)$
$\varphi_{v_2}(\St_G(v_2))$, which is a finitely generated branch group. 
It follows from Theorem \ref{thm:AlmostNormalFinitelyGenerated} that $\varphi_{v_2}(K_1)$ is finitely generated.

That means that one can find a finite $\tilde X_2\subseteq K$ such that $K$ is generated by $\tilde X_1\cup\tilde X_2$ together with $K_2=\rist_K(T_i\setminus (v_1\cup v_2))$.
Notice that $K_2$ is again normal in $K$.
Thus, by iterating this process, we obtain that $K$ is generated by $\tilde X_1\cup\dots\cup\tilde X_{d-1}$
together with $K_d=\rist_K(v_d)\cong\varphi_{v_d}(\rist_K(v_d))$ which is itself finitely generated.
So $K$ is finitely generated as desired.
\end{proof}

\section{Block subgroups of branch groups with the SIP}\label{section:BlockStructure}

In this section, we will show that finitely generated subgroups of tree-primitive branch groups with the subgroup induction property all contain a block subgroup of finite index. The proof is rather technical, but most of the work is done in Section~\ref{subsection:DependenceFunction}, where we introduce what we call the \emph{dependence function}. This serves as a tool that allows us, in Lemma~\ref{lemma:MinimalDependenceSetForAll}, to detect which vertices are ``dependent'' on each other, in the sense that a trivial action on the subtree rooted at one vertex implies a trivial action on the subtree rooted at the other. This does not imply that dependent vertices form part of a diagonal block subgroup however, since this dependence relation is not necessarily an equivalence relation. To solve this problem, we need to show that by passing to descendants of the original vertices, one can refine this relation into an equivalence relation and thus detect the vertices supporting the diagonal block subgroups, which is what we do in Lemmas~\ref{lemma:ExistsVertexWhoseProjectionIsFiniteIndex} to~\ref{lemma:BunchOfDiagonalBlocks}. Finally, in Section~\ref{subsection:ProofOfThm}, we apply these results to prove our main theorem.

\subsection{Full supporting sets and the dependence function}\label{subsection:DependenceFunction}

To study finitely generated subgroups of groups with the subgroup induction property more easily, let us introduce a new notion which we will call a \emph{full supporting set}.

\begin{defn}
Let $G\leq \Aut(X^*)$ be a finitely generated self-replicating group of automorphisms of $X^*$, and let $H\leq G$ be a finitely generated subgroup. A subset $F\subseteq X^{*}$ is said to be a \emph{full supporting set for $H$} if it is $H$-invariant and if there exists a transversal $T$ of $X^*$ with $F\subseteq T$ such that the following two conditions are satisfied:
\begin{enumerate}[label=(\roman*)]
\item $\varphi_v(\St_H(v))=G$ for all $v\in F$,
\item $|\varphi_v(\St_H(v))|<\infty$ for all $v\in T\setminus F$.
\end{enumerate}
\end{defn}

\begin{rem}
It follows from Proposition~\ref{prop:ExistNiceTransversal} that if $G\leq \Aut(X^*)$ is a finitely generated self-replicating group with the subgroup induction property, then every finitely generated subgroup of $G$ admits a full supporting set.
Note that the full supporting set of a finitely generated subgroup $H$ is empty if and only if $H$ is finite, in which case the main results of the present article hold trivially.
\end{rem}

We will now introduce a map defined on a full supporting set, which we will call the \emph{dependence function}.

\begin{defn}
Let $G\leq \Aut(X^*)$ be a finitely generated infinite self-replicating group of automorphisms of $X^*$, let $H\leq G$ be a finitely generated subgroup, and let $F\subseteq X^*$ be a non-empty full supporting set for $H$. The \emph{dependence function} on $F$ is the map $\delta_{H,F}\colon F \rightarrow \N$ given by
\[\delta_{H,F}(v)=\min\left\{|V| \mid v\in V\subseteq F, \varphi_v(\Srist_H(V))\ne 1\right\}.\]
\end{defn}

\begin{rem}
Let $F$ be a full supporting set of a finitely generated subgroup $H\leq G$.
From the definition, there exists a transversal $T$ of $X^*$ with $F\subseteq T$ such that $\varphi_v(\St_H(v))$ is finite for all $v\in T\setminus F$.
Therefore, if we denote by $K_v=\{h\in \St_H(v)\mid \varphi_v(h)=1\}$ the kernel of the application $\varphi_v$ restricted to $\St_H(v)$, then $K_v$ is a finite index subgroup of $H$ for all $v\in T\setminus F$.
Hence, $\rist_H(F)=\bigcap_{v\in T\setminus F}K_v$ is a finite index subgroup of $H$, and thus so is $\Srist_H(F)$.
Consequently, if $F$ is non-empty, then $\varphi_v(\Srist_H(F))$ is of finite index in $\varphi_v(\St_H(v))=G$, which implies that $\varphi_v(\Srist_H(F))\ne 1$, for all $v\in F$, so that $\delta_{H,F}$ is well-defined.
\end{rem}

\begin{defn}
Let $G\leq \Aut(X^*)$ be a finitely generated infinite self-replicating group of automorphisms of $X^*$, let $H\leq G$ be a finitely generated subgroup, and let $F\subseteq X^*$ be a non-empty full supporting set for $H$. Let $v\in F$ be any element. A subset $V\subseteq F$ will be called a \emph{minimal dependence set for $v$} if $v\in V$, $\varphi_v(\Srist_H(V))\ne 1$ and $|V|=\delta_{H,F}(v)$.
\end{defn}

Minimal dependence sets will be a key tool that will allow us to detect the supporting vertex sets of diagonal block subgroups. Let us now establish a few key facts about such sets.

\begin{lemma}\label{lemma:ProjectionsInMinimalDependenceSets}
Let $G\leq \Aut(X^*)$ be a finitely generated self-replicating group, let $H\leq G$ be a finitely generated subgroup and let $F\subseteq X^*$ be a non-empty full supporting set for $H$. Let $v\in F$ be any element, and let $V\subseteq F$ be a minimal dependence set for $v$. Then,
\begin{enumerate}[label=(\roman*)]
\item for every $h\in \Srist_H(V)$ and for every $w\in V$, $\varphi_v(h)\ne 1 \Rightarrow \varphi_w(h)\ne 1$,\label{item:NonTrivialvImpliesNonTrivialw}
\item for every $w\in V$, $\varphi_w(\Srist_H(V))$ is a non-trivial almost normal subgroup of $G$.\label{item:AlmostNormalSubgroup}
\end{enumerate}
\end{lemma}
\begin{proof}
\ref{item:NonTrivialvImpliesNonTrivialw} If there existed $h\in \Srist_H(V)$ and $w\in V$ such that $\varphi_v(h)\ne 1$ but $\varphi_w(h)=1$, then we would have $h\in \Srist_H(V\setminus \{w\})$ and thus $\varphi_v(\Srist_H(V\setminus\{w\}))\ne 1$. As $|V\setminus \{w\}|<|V|$, this would contradict the fact that $V$ is a minimal dependence set for $v$.

\ref{item:AlmostNormalSubgroup} Let $w\in V$ be any element. It follows from~\ref{item:NonTrivialvImpliesNonTrivialw} that $\varphi_w(\Srist_H(V))\ne 1$. Furthermore, as $\Srist_H(V)$ is normal in $\St_H(V)$, which is a finite index subgroup of $\St_H(w)$, we find by applying $\varphi_w$ that $\varphi_w(\Srist_H(V))$ must be normal in $\varphi_w(\St_H(V))$, which is of finite index in $\varphi_w(\St_H(w))=G$. Thus, $\varphi_w(\Srist_H(V))$ is a non-trivial almost normal subgroup of $G$.
\end{proof}

\begin{lemma}\label{lemma:DependenceSetsCommuteOrContained}
Let $G\leq \Aut(X^*)$ be a finitely generated self-replicating group, let $H\leq G$ be a finitely generated subgroup and let $F\subseteq X^*$ be a non-empty full supporting set for $H$. Let $v\in F$ be any element, let $V\subseteq F$ be a minimal dependence set for $v$ and let $W\subseteq F$ be another subset of $F$ containing $v$ and such that $\varphi_v(\Srist_H(W))\ne 1$. Then, either $V\subseteq W$ or $\varphi_v(\Srist_H(V))$ and $\varphi_v(\Srist_H(W))$ commute.
\end{lemma}
\begin{proof}
Suppose that $\varphi_v(\Srist_H(V))$ and $\varphi_v(\Srist_H(W))$ do not commute.  Since
\[[\Srist_H(V), \Srist_H(W)] \leq \Srist_H(V\cap W),\]
this implies that $\varphi_v(\Srist_H(V\cap W))\ne 1$. We obviously have $|V\cap W|\leq |V|$, but by the minimality of $V$, we must also have $|V|\leq |V\cap W|$, from which we conclude that $V=V\cap W$ and thus that $V\subseteq W$.
\end{proof}

As the next lemma establishes, in the case of branch groups, the value of the dependence function is constant on minimal dependence sets.

\begin{lemma}\label{lemma:MinimalDependenceSetForAll}
Let $G\leq \Aut(X^*)$ be a finitely generated self-replicating just infinite branch group, let $H\leq G$ be a finitely generated subgroup and let $F\subseteq X^*$ be a non-empty full supporting set for $H$. Let $v\in F$ be any element, and let $V\subseteq F$ be a minimal dependence set for $v$. Then, $V$ is a minimal dependence set for all $w\in V$. In particular, $\delta_{H,F}(w)=\delta_{H,F}(v)$ for all $w\in V$.
\end{lemma}
\begin{proof}
Let $w\in V$ be any element. We may assume that $w\ne v$, since there is nothing to prove in the case where $w=v$. It follows from Lemma~\ref{lemma:ProjectionsInMinimalDependenceSets} that $\varphi_w(\Srist_H(V))\ne 1$. Thus, to prove the result, we only need to show that $\delta_{H,F}(w)=\delta_{H,F}(v)$. Clearly, we have $\delta_{H,F}(w)\leq \delta_{H,F}(v)$. Let us suppose for the sake of contradiction that $\delta_{H,F}(w)<\delta_{H,F}(v)$, and let $W$ be a minimal dependence set for $w$.

For every $h\in \St_H(w)$, let us consider the set $h\cdot W$. By construction, we have $w\in h\cdot W$,
\[\varphi_w(\Srist_H(h\cdot W)) = \varphi_w(h)\varphi_w(\Srist_H(W))\varphi_w(h)^{-1}\ne 1\]
and $|h\cdot W| = |W| = \delta_{H,F}(w)$. In other words, for every $h\in \St_H(w)$, the set $h\cdot W$ is also a minimal dependence set for $w$.

For each $h\in \St_H(w)$, let us write $L_h=\varphi_w(\Srist_H(h\cdot W))$. By Lemma~\ref{lemma:ProjectionsInMinimalDependenceSets}~\ref{item:AlmostNormalSubgroup}, for every $h\in \St_H(w)$, the subgroup $L_h$ is almost normal in $G$. Now, if $h_1,h_2\in \St_H(w)$ are such that $[L_{h_1}, L_{h_2}]\ne 1$, then it follows from Lemma~\ref{lemma:DependenceSetsCommuteOrContained} that we have $h_1\cdot W = h_2\cdot W$. Consequently, for every $h_1,h_2\in \St_H(w)$, we either have $L_{h_1}=L_{h_2}$ or $[L_{h_1}, L_{h_2}]=1$.
As all the $L_h$ are conjugate in $G=\varphi_w(\St_H(w))$, and since there are only finitely many different sets of the form $h\cdot W$ with $h\in \St_H(w)$, there exist $h_1,\dots, h_k\in \St_H(w)$ such that $\{L_{h_1},\cdots, L_{h_k}\}$ form a conjugacy class in $G$.
Furthermore, since the $L_{h_i}$ are all almost normal in $G$ and pairwise commute, it follows from Corollary \ref{cor:CenterAlmostNormalIsTrivial} that
\[L_{h_1}\cdots L_{h_k} \cong \prod_{i=1}^{k}L_{h_i}.\]
Note that the product $L_{h_1}\cdots L_{h_k}$ is normal in $G$, and therefore of finite index, since $G$ is just infinite.

Let us write $K=\varphi_w(\Srist_H(V))$. Up to reordering the $h_i$, we can suppose that there exists $t\leq k$ such that $[L_{h_i}, K]\ne 1$ for all $1\leq i \leq t$ and $[L_{h_i}, K]= 1$ for all $t<i\leq k$. Notice that since $K\ne 1$ and is almost normal in $G$, it is infinite and has trivial center by Corollary~\ref{cor:CenterAlmostNormalIsTrivial}. It follows that $K$ cannot commute with all of the $L_{h_i}$, otherwise $K\cap L_{h_1}\cdots L_{h_k}$ would be in the center of $K$ and thus trivial, which is impossible since $K\cap L_{h_1}\cdots L_{h_k}$ is of finite index in the infinite group $K$, as $L_{h_1}\cdots L_{h_k}$ is of finite index in $G$.
Let us write $L_{\leq t} = L_{h_1}\cdots L_{h_t}$ and $L_{>t}=L_{h_{t+1}}\cdots L_{h_k}$. Then, both of these subgroups are almost normal in $G$, and are complement to each other since we have $L_{\leq t}L_{>t} \cong L_{\leq t}\times L_{>t}$. Since $[K, L_{>t}]=1$ by construction, we have by Proposition~\ref{prop:CommutingAlmostNormalLiesInComplement} that $K\cap L_{\leq t}$ is of finite index in $K$.

One can define a homomorphism
\begin{align*}
\alpha\colon K=\varphi_w(\Srist_H(V))&\rightarrow \varphi_v(\Srist_H(V))\\
\varphi_w(h) &\mapsto \varphi_v(h).
\end{align*}
This is well-defined, because if $h_1, h_2\in \Srist_H(V)$ are such that $\varphi_w(h_1)=\varphi_w(h_2)$, then $\varphi_w(h_1h_2^{-1})=1$, which implies by Lemma~\ref{lemma:ProjectionsInMinimalDependenceSets}~\ref{item:NonTrivialvImpliesNonTrivialw} that $\varphi_v(h_1h_2^{-1})=1$ and thus that $\varphi_v(h_1)=\varphi_v(h_2)$. Furthermore, it is clear from its definition that this homomorphism is surjective. Consequently, $\alpha(K\cap L_{\leq t})$ is a subgroup of finite index of $\varphi_v(\Srist_H(V))$, and therefore is non-trivial, since $\varphi_v(\Srist_H(V))$ must be infinite, as follows from Lemma~\ref{lemma:ProjectionsInMinimalDependenceSets}~\ref{item:AlmostNormalSubgroup} and Corollary~\ref{cor:CenterAlmostNormalIsTrivial}.

Since $\alpha(K\cap L_{\leq t})\ne 1$, there must exist $h\in \Srist_H(V)$ such that $1\ne \varphi_v(h)\in \alpha(K\cap L_{\leq t})$ and $\varphi_w(h)\in L_{\leq t}$. Now, by the definition of $L_{\leq t}$, there must exist elements $g_i\in \Srist_H(h_i\cdot W)$, $1\leq i \leq t$, such that $\varphi_w(g)=\varphi_w(h)$, where $g=g_1\cdots g_t$. Since each $g_i\in \Srist_H(h_i\cdot W)$, we must have $g\in \Srist_H\left(\bigcup_{i=1}^{t}h_i\cdot W\right)$. Remember that by definition, for every $1\leq i \leq t$, the subgroups $\varphi_w(\Srist_H(h_i\cdot W))$ and $\varphi_w(\Srist_H(V))$ do not commute. Therefore, it follows from Lemma~\ref{lemma:DependenceSetsCommuteOrContained} that $h_i\cdot W \subseteq V$ for all $1\leq i \leq t$ and thus that $\bigcup_{i=1}^{t}h_i\cdot W \subseteq V$. Furthermore, none of the $h_i\cdot W$ contain $v$, otherwise we would get a contradiction to the minimality of $V$ as a dependence subset of $v$, since $|h_i\cdot W|=|W|<|V|$. Putting these facts together, we find that $g\in \Srist_H\left(\bigcup_{i=1}^{t}h_i\cdot W\right)\leq \Srist_H(V)$, with $\varphi_v(g)=1$ and $\varphi_w(g)=\varphi_w(h)$. Consequently, we have $hg^{-1}\in \Srist_H(V)$ with $\varphi_v(hg^{-1})=\varphi_v(h)\ne 1$ and $\varphi_w(hg^{-1})=\varphi_w(h)\varphi_w(h)^{-1}=1$. In other words, we have an element $hg^{-1}\in \Srist_H(V\setminus\{w\})$ such that $\varphi_v(hg^{-1})\ne 1$, which contradicts the fact that $V$ is a minimal dependence set for $v$. Having reached a contradiction, we conclude that our initial assumption that $\delta_{H,F}(w)<\delta_{H,F}(v)$ was wrong and therefore that $\delta_{H,F}(w)=\delta_{H,F}(v)$.
\end{proof}

As it turns out, for branch groups with tree-primitive actions, one can associate a diagonal block subgroup to every minimal dependence set, as we will show in the next couple of lemmas. Note however that the supporting set of this diagonal block need not be the same as the minimal dependence set.

\begin{lemma}\label{lemma:ExistsVertexWhoseProjectionIsFiniteIndex}
Let $G\leq \Aut(X^*)$ be a finitely generated self-replicating just infinite branch group acting tree-primitively on $X^*$, let $H\leq G$ be a finitely generated subgroup and let $F\subseteq X^*$ be a non-empty full supporting set for $H$. Let $v\in F$ be any element, and let $V=\{v_1,\dots, v_k\}\subseteq F$ be a minimal dependence set for $v$. Then, there exists a unique $w\in X^*$ such that $\varphi_v(\Srist_H(V))\leq \rist_G(w)$ and $\varphi_{vw}(\Srist_H(V))$ is a finite index subgroup of $G$.
\end{lemma}
\begin{proof}
By Lemma~\ref{lemma:ProjectionsInMinimalDependenceSets}, $\varphi_v(\Srist_H(V))$ is a non-trivial almost normal subgroup of~$G$, and it follows from Lemma~\ref{lemma:DependenceSetsCommuteOrContained} and the fact that $\varphi_v(\St_H(v))=G$ that the conjugates of $\varphi_v(\Srist_H(V))$ by any element of $G$ are either equal to $\varphi_v(\Srist_H(V))$ or commute with it.

By Theorem~\ref{thm:StructureAlmostNormalSubgroups}, there exist $n\in \N$ and some $W\subseteq X^n$ such that $\rist_G(n)'\cap \rist_G(W)\leq \varphi_v(\Srist_H(V)) \leq \rist_G(W)$.
Since $g\rist_G(W)g^{-1} = \rist_G(g\cdot W)$, we have $\rist_G(n)'\cap \rist_G(g\cdot W)\leq g\varphi_v(\Srist_H(V))g^{-1} \leq \rist_G(g\cdot W)$ for all $g\in G$.
If $g\varphi_v(\Srist_H(V))g^{-1}=\varphi_v(\Srist_H(V))$, then we have $\rist_G(n)'\cap \rist_G(g\cdot W) \leq \rist_G(W)$ and $\rist_G(n)'\cap \rist_G(W) \leq \rist_G(g\cdot W)$, which implies that $g\cdot W = W$.
On the other hand, if $g\varphi_v(\Srist_H(V))g^{-1}\ne\varphi_v(\Srist_H(V))$, then they must commute, and since we have
\[\rist_G(n)'\cap \rist_G(g\cdot W \cap W)\leq g\varphi_v(\Srist_H(V))g^{-1}\cap\varphi_v(\Srist_H(V)),\]
we conclude from the fact that $\rist_G(g\cdot W \cap W)$ is almost normal in $G$ and Corollary~\ref{cor:CenterAlmostNormalIsTrivial} that $g\cdot W \cap W=\emptyset$.
Therefore, for any $g\in G$, either $g\cdot W = W$ or $g\cdot W \cap W = \emptyset$. It follows that there must exist $g_1,\dots, g_l\in G$ such that for every $g\in G$, there exist some $i\in \{1,\dots, l\}$ such that $g\cdot W = g_i\cdot W$ and that $g_i\cdot W \cap g_j \cdot W = \emptyset$ whenever $i\ne j$. Using additionally the fact that $G$ acts transitively on $X^n$, we find that the collection $\{g_i\cdot W\}_{i=1}^{l}$ must form a $G$-invariant partition of $X^n$. Since the action of $G$ on $X^*$ is tree-primitive, there must exist some $w\in X^*$ such that $|w|\leq n$ and $W=wX^{n-|w|}$.

It follows that we have $\varphi_v(\Srist_H(V))\leq \rist_G(w)$ and $\rist_G(n)'\cap \rist_G(w) = \rist_G(n)'\cap \rist_G(W) \leq \varphi_v(\Srist_H(V))$. Consequently,
\begin{align*}
\varphi_{vw}(\Srist_H(V)) &= \varphi_w(\varphi_v(\Srist_H(V)))\\
	&\geq \varphi_w(\rist_G(n)'\cap \rist_G(w)) = \varphi_w(\rist_G(n)').
\end{align*}
Since $\rist_G(n)'$ is a non-trivial normal subgroup of the just infinite group $G$, it is of finite index in $G$.
It follows that $\rist_G(n)'$ is of finite index in $\St_G(w)$, and thus that $\varphi_w(\rist_G(n)')$ is of finite index in $\varphi_w(\St_G(w))=G$, which implies that $\varphi_{vw}(\Srist_H(v))$ is a finite index subgroup of $G$.

To prove that $w$ is unique, let $w'\in X^*$ be such that $\varphi_v(\Srist_H(V))\leq \rist_G(w')$. Then, $w'$ must be comparable to $w$. If $w'<w$, then we can write $w=w'u$ with $u$ different from the empty word. In this case, we have $\varphi_{vw'}(\Srist_H(V))\leq \varphi_{w'}(\rist_G(w)) \leq \rist_G(u)$, which is not of finite index in $G$. If $w<w'$, then a symmetric argument shows that $\varphi_{vw}(\Srist_H(V))$ cannot be of finite index in $G$. Thus, $w$ must be unique.
\end{proof}

\begin{lemma}\label{lemma:DiagonalSubgroupFromMinimalDependenceSet}
Let $G\leq \Aut(X^*)$ be a finitely generated self-replicating just infinite branch group acting tree-primitively on $X^*$, let $H\leq G$ be a finitely generated subgroup and let $F\subseteq X^*$ be a non-empty full supporting set for $H$. Let $V=\{v_1,\dots, v_k\}\subseteq F$ be a minimal dependence set. Then, there exist unique elements $w_i\in v_iX^*$ for all $1\leq i \leq k$ such that $\Srist_H(V)$ is a diagonal block subgroup of $G$ with supporting vertex set $W=\{w_1,\dots, w_k\}$.
\end{lemma}
\begin{proof}
By Lemma~\ref{lemma:ExistsVertexWhoseProjectionIsFiniteIndex}, for every $i\in \{1,\dots, k\}$, there exists a unique $u_i\in X^*$ such that $\varphi_{v_i}(\Srist_H(V))\leq \rist_G(u_i)$ and $\varphi_{v_iu_i}(\Srist_H(V))$ is of finite index in $G$. Let us write $w_i=v_iu_i$, and let $W=\{w_1,\dots, w_k\}$. We want to show that $\Srist_H(V)$ is a diagonal block subgroup of $G$ with supporting vertex set $W$. The uniqueness will follow directly from the uniqueness of the elements $u_i$ guaranteed by Lemma~\ref{lemma:ExistsVertexWhoseProjectionIsFiniteIndex}.

To simplify the notation, let us write $L=\Srist_H(V)$. The first thing to show is that $L\leq \Srist_H(W)$. For this, let us first observe that $L\leq \St_G(W)$. Indeed, for every $g\in L$ and for every $w_i\in W$, we have
\[g\cdot w_i = g\cdot v_iu_i = (g\cdot v_i)(\varphi_{v_i}(g))\cdot u_i = v_iu_i=w_i\]
where we have used the fact that $g\in L\leq \St_G(V)$ and $\varphi_{v_i}(g)\in \rist_G(u_i)\leq \St_G(u_i)$. Now, to prove that $L\leq \rist_H(W)$, let $z\in X^*$ be an element that is not comparable with any element of $W$. If $z$ is not comparable with any element of $V$, then we already know that $L\cdot z=z$, since $L\leq \rist_H(V)$. If $z$ is comparable with some element $v_i\in V$, then we must have $v_i\leq z$. Indeed, the only other possibility is $z\leq v_i$, but then we would have $z\leq v_i\leq w_i$, which contradicts our assumption that $z$ is not comparable to any element of $W$. Thus, there exists some $y\in X^*$ such that $z=v_iy$. Furthermore, this $y$ cannot be comparable with $u_i$, otherwise $z$ and $w_i$ would be comparable. This implies that $\rist_G(u_i)\cdot y=y$, and since
\[L\cdot z = L\cdot v_iy = v_i (\varphi_{v_i}(L)\cdot y)\]
with $\varphi_{v_i}(L)\leq \rist_G(u_i)$, we conclude that $L\cdot z=z$. This finishes showing that $L\leq \Srist_H(W)$.

The second thing we need to prove is that $\varphi_{w_i}(L)$ is of finite index in $G$ for all $1\leq i \leq k$. This, however, is immediately verified by our construction of the $w_i$.

Lastly, to prove that $L$ is a diagonal block subgroup with supporting vertex set $W$, we need to show that for every $1\leq i \leq k$, the map $\varphi_{w_i}$ is injective on $L$. Let us fix $i\in \{1,\dots, k\}$, and let $g\in L$ be an element such that $\varphi_{w_i}(g)=1$. Then, we claim that $\varphi_{v_i}(g)=1$. Indeed, since $\varphi_{v_i}(L)\leq \rist_G(w_i)$, we see that $\varphi_{v_i}(g)$ fixes every element of $X^*\setminus w_iX^*$, so by our assumption that $\varphi_{v_i}(g)$ also fixes every element of $w_iX^*$, we conclude that $\varphi_{v_i}(g)=1$. Let $v_j\in V$ be any element. By Lemma~\ref{lemma:MinimalDependenceSetForAll}, the set $V$ is a minimal dependence set for $v_j$, and therefore, by Lemma~\ref{lemma:ProjectionsInMinimalDependenceSets}~\ref{item:NonTrivialvImpliesNonTrivialw}, we must also have $\varphi_{v_j}(g)=1$. As $v_j$ was arbitrary, we see that $g$ fixes all elements of $VX^*$, and since $g\in \Srist_H(V)$, it must also fix all elements outside of $VX^*$, which means that $g=1$. Therefore, $\varphi_{v_i}$ is injective on $L$, which concludes the proof.
\end{proof}

\begin{lemma}\label{lemma:MinimalDependenceSetsInvariantByH}
Let $G\leq \Aut(X^*)$ be a finitely generated self-replicating just infinite branch group acting tree-primitively on $X^*$, let $H\leq G$ be a finitely generated subgroup and let $F\subseteq X^*$ be a non-empty
full supporting set for $H$. Let us denote by $\VV$ the set of all minimal dependence sets for $H$. For any $V\in \VV$, let $W_V\subseteq X^*$ denote the unique set described in Lemma~\ref{lemma:DiagonalSubgroupFromMinimalDependenceSet}, and let us denote by $W_{\VV}=\{W_V \subseteq X^* \mid V\in \VV\}$. Then, both $\VV$ and $W_{\VV}$ are finite $H$-invariant subsets of $\PP(X^*)$.
\end{lemma}
\begin{proof}
Let $V\in \VV$ be a minimal dependence set, let $h\in H$ be any element, and let $v\in V$ be any element. Since $V$ is a minimal dependence set for $v$ by Lemma~\ref{lemma:MinimalDependenceSetForAll}, we have $\varphi_v(\Srist_H(V))\ne 1$. Since $\Srist_H(h\cdot V) = h\Srist_H(V)h^{-1}$, we must have $\varphi_{h\cdot v}(\Srist_H(h\cdot V))\ne 1$. If $\delta_{H,F}(h\cdot v)\ne |h\cdot V|$, then there must exist some minimal dependence set $W$ for $h\cdot v$ with $|W|<|V|$. However, in this case, the same argument as above would yield $\varphi_v(\Srist_H(h^{-1}\cdot W))\ne 1$ and thus $\delta_{H,F}(v)\leq |W|<|V|=\delta_{H,F}(v)$, which is absurd. We conclude that $h\cdot V$ must be a minimal dependence set for $h\cdot v$. This proves that $\VV$ is $H$-invariant.

To prove that $W_{\VV}$ is also $H$-invariant, it suffices to notice that $h\cdot W_V = W_{h\cdot V}$ for all $V\in \VV$ and $h\in H$. For this, it suffices to show that if $\Srist_H(V)$ is a diagonal block subgroup of $G$ with supporting vertex set $W_V$, then $\Srist_H(h\cdot V)$ is a diagonal block subgroup of $G$ with supporting vertex set $W_{h\cdot V}$. However, this is almost immediate from the fact that $\Srist_H(h\cdot V) = h\Srist_H(V)h^{-1}$, since the properties defining diagonal block subgroups are invariant under conjugation.

Finiteness of both $\VV$ and $W_\VV$ follows from the fact that these two sets are in bijection and that $\VV\subseteq \PP(F)$ with $F$ finite.
\end{proof}

Using the previous results, one can show that in groups with the SIP, one can always find for any given finitely generated subgroup a full supporting set that can be nicely partitioned in such a way that the rigid stabilisers of elements of this partition are diagonal subgroups.

\begin{lemma}\label{lemma:BunchOfDiagonalBlocks}
Let $G\leq \Aut(X^*)$ be a finitely generated self-replicating branch group with the subgroup induction property acting tree-primitively on $X^*$, and let $H\leq G$ be a finitely generated subgroup of $G$. Then, there exist a full supporting set $F\subseteq X^*$ for $H$ and a partition $F=\bigsqcup_{i=1}^{k}F_i$ such that for every $1\leq i \leq k$, $\Srist_H(F_i)$ is a diagonal block subgroup with supporting vertex set $F_i$.
\end{lemma}
\begin{proof}
Since $G$ has the SIP, it follows from Proposition~\ref{prop:ExistNiceTransversal} that there exists a full supporting set $F'\subseteq X^*$ for $H$. For any minimal dependence set $V\subseteq F'$, Lemma~\ref{lemma:DiagonalSubgroupFromMinimalDependenceSet} gives us a subset $W_V\subseteq X^*$ such that $\Srist_H(V)$ is a diagonal block subgroup with supporting vertex set $W_V$ and such that for every $v\in V$, there exists a unique $w\in W$ with $v\leq w$.

Let us first show that if $V_1, V_2\subseteq F'$ are two distinct minimal dependence sets for $H$, then any pair of elements $w_1\in W_{V_1}$ and $w_2\in W_{V_2}$ are incomparable (notice that in particular, this implies that $W_{V_1}\cap W_{V_2}=\emptyset$). Indeed, let $w_1\in W_{V_1}$ and $w_2\in W_{V_2}$ be any two elements. Then, there exist unique elements $v_1\in V_1$ and $v_2\in V_2$ such that $v_1\leq w_1$ and $v_2\leq w_2$. Suppose that $w_1$ and $w_2$ are comparable. Then, this implies that $v_1$ and $v_2$ must also be comparable. Since $v_1,v_2\in F'$, which is a set of incomparable elements by definition, we must have $v_1=v_2$. Therefore, to prove our claim, it suffices to show that for any $v\in V_1\cap V_2$, the unique elements $w_1\in W_{V_1}$ and $w_2\in W_{V_2}$ satisfying $w_i\geq v$ are incomparable. By Lemma~\ref{lemma:DependenceSetsCommuteOrContained}, if we have $v\in V_1\cap V_2$ with $V_1\ne V_2$, then $\varphi_v(\Srist_H(V_1))$ and $\varphi_v(\Srist_H(V_2))$ commute. Since $\varphi_v(\Srist_H(V_i))$ is a finite index subgroup of $\rist_G(u_i)$ for $i=1,2$, where $w_i=vu_i$, it follows from Corollary~\ref{cor:CenterAlmostNormalIsTrivial} that $w_1$ and $w_2$ are incomparable, as otherwise the almost normal subgroups $\varphi_v(\Srist_H(V_1))$ and $\varphi_V(\Srist_H(V_2))$ would intersect non-trivially and thus have non-trivial center.

If we denote by $\VV$ the collection of all minimal dependence sets for $H$, it follows from the above that $W_{\VV}=\{W_V\}_{V\in \VV}$ is a collection of pairwise disjoint sets, so that it forms a partition of their union. Let us write $F=\bigsqcup_{V\in \VV} W_V$. As any two elements in any given $W_V$ are incomparable, and as any two elements in any distinct sets $W_{V_1}$ and $W_{V_2}$ are also incomparable by the above, $F$ is a set of incomparable elements. As $W_{\VV}$ is a finite $H$-invariant set by Lemma~\ref{lemma:MinimalDependenceSetsInvariantByH}, $F$ must be a finite $H$-invariant set. We want to show that $F$ is a full supporting set for $H$. Let $T'\supseteq F'$ be a transversal such that $|\varphi_v(\St_H(v))|<\infty$ if $v\in T'\setminus F'$, and let us define $T=(T'\setminus F') \cup F$. 
By construction, we have $|\varphi_v(\St_H(v))|<\infty$ for all $v\in T\setminus F$. Using the fact that $F'$ is a full supporting set for $H$ and that $G$ is self-replicating, we also have $\varphi_v(\St_H(v))=G$ for all $v\in F$, since for all $v\in F$, there exists $v'\in F'$ such that $v'\leq v$ and $\varphi_{v'}(\St_H(v'))=G$. To show that $F$ is a full supporting set for $H$, it thus only remains to show that $T$ is a transversal. For this, we need to show that for every $w\in X^*$, there exists some $v\in T$ such that $v$ and $w$ are comparable. Let $w\in X^*$ be any element. Since $T'$ is a transversal, there must exist some $v'\in T'$ such that $v'$ and $w$ are comparable. If $v'\in T'\setminus F'$, then $v'\in T$, so $w$ is comparable to some element of $T$. If $v'\in F'$, then there exist some $v\in F$ such that $v'\leq v$. As $\varphi_{v'}(\St_H(v'))=G$, we know that $H$ acts transitively on the set $v'X^{|v|-|v'|}$. Consequently, there must exist some $h\in \St_H(v')$ such that $h\cdot v$ is comparable with $w$. Now, since $F$ is $H$-invariant, we conclude that $w$ is comparable with an element of $F$, which finishes showing that $T$ is a transversal and thus that $F$ is a full supporting set for $H$.

To finish the proof, we only need to show that $\Srist_H(W_V)$ is  diagonal block subgroup with supporting set $W_V$ for every $V\in \VV$. Let us fix $V\in \VV$. We know by Lemma~\ref{lemma:DiagonalSubgroupFromMinimalDependenceSet} that $\Srist_H(V)$ is a diagonal block subgroup with supporting vertex set $W_V$. In particular, this means that $\Srist_H(V)\leq \Srist_H(W_V)$. As $W_V\subseteq VX^*$, we have $\Srist_H(W_V)\leq \Srist_H(V)$, so that $\Srist_H(V)=\Srist_H(W_V)$. Thus, $\Srist_H(W_V)$ is a diagonal block subgroup with supporting vertex set $W_V$.
\end{proof}

\subsection{Block subgroups in groups with the SIP}\label{subsection:ProofOfThm}

We are now in position to show the main result of this section, namely that every finitely generated subgroup of a group with the SIP acting tree-primitively on $X^*$ contains a block subgroup of finite index.

\begin{thm}\label{thm:BlockSubgroups}
Let $G\leq \Aut(X^*)$ be a finitely generated self-replicating branch group with the subgroup induction property acting tree-primitively on $X^*$, and let $H\leq G$ be a finitely generated subgroup of $G$. Then, $H$ is virtually a block subgroup. More precisely, there exist a full supporting set $F\subseteq X^*$ for $H$ and a partition $P=\{F_i\}_{i=1}^{k}$ of $F$ such that $B=B_1\cdots B_k$ is a block subgroup of finite index in $H$ with supporting partition $P$, where $B_i=\Srist_H(F_i)$ for $i=1,\dots, k$.
\end{thm}
\begin{proof}
By Lemma~\ref{lemma:BunchOfDiagonalBlocks}, there exists a full supporting set $F\subseteq X^*$ for $H$ and a partition $P=\{F_i\}_{i=1}^k$ of $F$ such that for every $1\leq i \leq k$, the subgroup $B_i=\Srist_H(F_i)$ is a diagonal block subgroup with supporting vertex set $F_i$.
Notice that if $i\ne j$, then $[B_i,B_j]=1$ and $B_i\cap B_j=1$, since $F_i\cap F_j=\emptyset$. Therefore, we can define the subgroup
\[B=B_1\cdots B_k \cong \prod_{i=1}^{k}B_i,\]
which is a block subgroup with supporting partition $P$.

To prove the theorem, we only need to show that $B$ is of finite index in $H$. It follows from the definition of a full supporting set that $\Srist_H(F)$ is of finite index in $H$. For every $v\in F$, let us define
\[L_v = \{h\in \Srist_H(F) \mid \varphi_v(h)\in \varphi_v(B)\}.\]
As $\varphi_v(B)=\varphi_v(B_i)$ for some $1\leq i \leq k$, with $B_i$ a diagonal block subgroup, $\varphi_v(B)$ is a finite index subgroup of $G$. It follows that $L_v$ is a finite index subgroup of $\Srist_H(F)$ and thus of $H$. Using the fact that $F$ is finite, we deduce that $L=\bigcap_{v\in F}L_v$ is a finite index subgroup of $H$.

We want to show that $L=B$. For this, let $g\in L$ be any element. For every $1\leq i \leq k$, we can find an element $b_i\in B_i$ such that $\varphi_v(b_i)=\varphi_v(g)$ for all $v\in F_i$. Indeed, let $v\in F_i$ be any element. By definition of $L$, there exists $b_i\in B_i$ such that $\varphi_v(b_i)=\varphi_v(g)$. Now, if there existed $w\in F_i$ such that $\varphi_w(b_i)\ne \varphi_w(g)$, then we would have $gb_i^{-1}\in \Srist_H(F)$ with $\varphi_w(gb_i^{-1})\ne 1$, but $\varphi_v(gb_i^{-1})=1$. As $B_i$ is a normal subgroup of $\Srist_H(F)$, for every $h\in B_i$ we would have $[h, gb_i^{-1}]\in B_i$ with $\varphi_v([h,gb_i^{-1}])=1$. By the injectivity of $\varphi_v$ on $B_i$, this would imply that $[h,gb_i^{-1}]=1$ for every $h\in B_i$ and thus that $\varphi_w(gb_i^{-1})$ would be a non-trivial element of $\varphi_w(B_i)$ that commutes with every element of $\varphi_w(B_i)$. Since $\varphi_w(B_i)$ is an almost normal subgroup of $G$, this would contradict Corollary~\ref{cor:CenterAlmostNormalIsTrivial}. Therefore, we can conclude that $\varphi_w(b_i)=\varphi_w(g)$ for all $w\in F_i$. Notice that since $B_i=\Srist_H(F_i)$, we have $\varphi_w(b_i)=1$ for all $w\in F\setminus F_i$. Therefore, if we define $b=\prod_{i=1}^{k}b_i$, we have $\varphi_v(b)=\varphi_v(g)$ for all $v\in F$, so that $\varphi_v(gb^{-1})=1$ for all $v\in F$. Therefore, $gb^{-1}$ fixes every vertex comparable with an element of $F$. Since $gb^{-1}\in \Srist_H(F)$, it also fixes every vertex that is not comparable with an element of $F$, which means that $gb^{-1}$ fixes every element of $X^*$. Thus, $g=b\in B$, which shows that $B=L$ and thus is of finite index in $H$.
\end{proof}

\begin{rem}
It is not difficult to obtain as a corollary that all finitely generated subgroups of a group satisfying the hypotheses of Theorem~\ref{thm:BlockSubgroups} are closed in the profinite topology. In other words, such a group is LERF. However, we will not present a proof here, since one can show directly (at the price of having a more complicated proof) that a self-replicating branch group with the subgroup induction property is LERF, without any assumption on the tree-primitivity of the action (see~\cite{FrancoeurLeemann20}).
\end{rem}

\begin{rem}
Note that in Theorem~\ref{thm:BlockSubgroups}, one could remove the assumption that $G$ is finitely generated. Indeed, by Corollary~4.5 of~\cite{GrigorchukLeemannNagnibeda21}, a branch group with the subgroup induction property is either finitely generated or locally finite, and the conclusion of Theorem~\ref{thm:BlockSubgroups} obviously holds for locally finite groups, although it is not interesting in this case.
The same remark also holds for Theorem~\ref{thm:RegularBlockSubgroup} below.
\end{rem}

\begin{rem}
A careful reader might have spotted the fact that in the proof of Theorem~\ref{thm:BlockSubgroups} the hypotheses of $H$ being finitely generated and of $G$ being self-replicating and having the subgroup induction property are used only twice. Once to show the existence of a full supporting set for $H$ in Proposition~\ref{prop:ExistNiceTransversal}, and a second time in Lemma~\ref{lemma:MinimalDependenceSetForAll} to ensure that $G$ is just infinite.
That is, we have proved the following more precise version of Theorem~\ref{thm:BlockSubgroups}.

Let $H$ be a subgroup of a finitely generated group $G\leq\Aut(X^*)$.
Then for the properties
\begin{enumerate}
\item
$H$ is finitely generated, \label{Property1}
\item
There exists a full supporting set for $H$. In other words: there exists a transversal $T$ of $X^*$ such that the sections of $H$ along $T$ are either equal to $G$ or finite, \label{Property2}
\item
$H$ is virtually a block subgroup, \label{Property3}
\end{enumerate}
we have the following implications.
\ref{Property3} always implies \ref{Property1}.
If $G$ is branch, then \ref{Property2} implies \ref{Property1} by Lemma \ref{lem:NiceTransversalFG}.
If $G$ is self-replicating with the subgroup induction property, then \ref{Property1} implies  \ref{Property2} by Proposition~\ref{prop:ExistNiceTransversal}.
Finally, if $G$ is a just infinite branch group acting tree-primitively on $X^*$, then  \ref{Property2} implies  \ref{Property3}, see the proof of Theorem~\ref{thm:BlockSubgroups}.
In particular, if $G$ is a finitely generated self-replicating branch group with the subgroup induction property and acting tree-primitively on $X^*$, then the Properties  \ref{Property1} to  \ref{Property3} are all equivalent.
\end{rem}

Using the previous theorem, one can obtain the slightly stronger result below, which will be useful later on.

\begin{cor}\label{cor:BlockSubgroupNiceProjections}
Let $G\leq \Aut(X^*)$ be a finitely generated self-replicating branch group with the subgroup induction property acting tree-primitively on $X^*$, and let $H\leq G$ be a finitely generated subgroup of $G$. Then, there exist a full supporting set $F\subseteq X^*$ for $H$ and a partition $P=\{F_i\}_{i=1}^{k}$ of $F$ such that
\begin{enumerate}[label=(\roman*)]
\item $B=B_1\cdots B_k$ is a block subgroup of finite index in $H$ with supporting partition $P$, where $B_i=\Srist_H(F_i)$,
\item $\varphi_v(\St_H(F_i))=G$ for all $i\in\{1,\dots, k\}$ and for all $v\in F_i$.
\end{enumerate}
\end{cor}
\begin{proof}
Applying Theorem~\ref{thm:BlockSubgroups} to $H$, one finds a full supporting set $F\subseteq X^*$ of $H$ and a partition $P=\{F_i\}_{i=1}^{k}$ of $F$ such that $B=B_1\cdots B_k$ is a block subgroup of finite index in $H$ with supporting partition $P$, where $B_i=\Srist_H(F_i)$. By definition of a full supporting set, we have $\varphi_v(\St_H(v))=G$ for all $v\in F$. It could happen, however, that $\varphi_v(\St_H(F_i))<G$ for some $i\in \{1,\dots, k\}$ and some $v\in F_i$.

To solve this problem, let us consider the subgroup $L=\St_H(F)$, which is of finite index in $H$. Applying Theorem~\ref{thm:BlockSubgroups} to $L$ yields a full supporting set $F'\subseteq X^*$ for $L$ and a partition $P'=\{F'_i\}_{i=1}^{k'}$ such that $C=C_1\cdots C_{k'}$ is a block subgroup of finite index in $L$ with supporting partition $P'$, where $C_i=\Srist_L(F'_i)$. Since $L$ is of finite index in $H$, the set $F'$ is also a full supporting set for $H$ and $C$ is also a finite index subgroup of $H$.

Notice that for every $v\in F'$, there must exist some $u\in F$ such that $u\leq v$. Indeed, since both $F'$ and $F$ are full supporting sets for $H$, every $v\in F'$ must be comparable with an element of $F$, so there must exist some $u\in F$ such that either $u\leq v$ or $v<u$. The case $v<u$ is impossible, since $L$ fixes $u$ by definition, but $\varphi_v(\St_L(v))=G$, which implies that $\St_L(v)$ does not fix $u$, as $G$ acts spherically transitively on $X^*$.

Similarly, for every $u\in F$, there must exist some $v\in F'$ comparable with $u$, and by the same argument, we must have $u\leq v$.
Therefore, if there exists $h\in H$ such that $h\cdot u\ne u$, then we also have $h\cdot v \ne v$.
It follows that $\St_H(F')\leq \St_H(F)=L$.
Consequently, we have $\Srist_H(F'_i) = \Srist_L(F'_i)=C_i$ for all $1\leq i \leq k$, since $\Srist_H(F'_i)\leq \St_H(F')\leq L$.
Thus, to prove the result, it only remains to show that $\varphi_v(\St_H(F'_i))=G$ for all $i\in \{1,\dots, k'\}$ and all $v\in F'$.
Notice that since $\varphi_v(\St_L(F_i'))\leq \varphi_v(\St_H(F_i')) \leq G$, it suffices to show that $\varphi_v(\St_L(F_i'))=G$.

Now, let $v_1,v_2\in F_i'$ be two elements belonging to the same set $F'_i$ (for some $i\in\{1,\dots, k'\}$) and let $u_1,u_2\in F$ be the two unique elements such that $u_1\leq v_1$ and $u_2\leq v_2$. We claim that if $v_1\ne v_2$, then $u_1\ne u_2$. Suppose on the contrary that $u_1=u_2=u$, and let us write $v_1=uw_1$, $v_2=uw_2$ for some $w_1,w_2\in X^*$. Let $C_i\leq C$ be the diagonal subgroup of $C$ with supporting vertex set $F'_i$. As $C$ is of finite index in $H$, we know that $\varphi_u(C)$ is of finite index in $G$, and since $C_i$ is normal in $C$, the subgroup $\varphi_u(C_i)$ is an almost normal subgroup of $G$.
By Theorem~\ref{thm:StructureAlmostNormalSubgroups}, there exist $n\in \N$ and $W\subseteq X^n$ such that $\rist_G(n)'\cap \rist_G(W)\leq \varphi_u(C_i)\leq \rist_G(W)$.
Since $\varphi_{w_j}(\varphi_u(C_i)) = \varphi_{v_i}(C_i) \ne 1$ for $j=1,2$, we deduce that $w_1$ and $w_2$ are comparable with some elements of $W$.
It follows, using the fact that $\rist_G(m)'\leq \rist_G(n)'$ for all $m\geq n$, that there exist some $w_1'\geq w_1$ and $w_2'\geq w_2$ such that $\rist_G(w_1')', \rist_G(w_2')'\leq \rist_G(n)'\cap \rist_G(W) \leq \varphi_u(C_i)$. Since $w_1\ne w_2$, this implies that $\rist_{\varphi_u(C_i)}(w_1)$ is contained in the kernel of the map $\varphi_{w_2}$, which in turn implies that $\varphi_{v_2}$ is not injective on $C_i$, a contradiction to the fact that $C_i$ is a diagonal block subgroup on~$F'_i$.

Using the above, we can conclude that $\St_L(F_i') = \St_L(v)$ for all $i\in \{1,\dots, k'\}$ and all $v\in F_i'$. Indeed, it follows from Lemma~\ref{lemma:MinimalDependenceSetsInvariantByH} and the fact that the supporting vertex set of diagonal subgroups are minimal dependence sets that $P'$ is an $L$-invariant partition of $F'$. In particular, $\St_L(v)\cdot F_i'=F_i'$ for all $i\in \{1,\dots, k'\}$ and all $v\in F_i'$. Let $v_1,v_2\in F_i'$ be any pair of elements, let $g\in \St_L(v_1)$ be arbitrary and let $u_2\in F$ be such that $u_2\leq v_2$. Since $\St_L(v_1)\cdot F_i'=F_i'$, we have $g\cdot v_2\in F_i'$, and since $L\leq \St_H(F)$, we also have $g\cdot u_2=u_2$. By the above, this implies that $g\cdot v_2=v_2$. We conclude that $\St_L(v_1)$ fixes every element of $F_i'$.

To finish the proof, we simply need to notice that for all $i\in \{1,\dots, k'\}$ and all $v\in F_i'$, we have $\varphi_v(\St_L(F_i')) = \varphi_v(\St_L(v)) = G$.
\end{proof}

\section{Regular block subgroups of groups with the SIP}

Theorem~\ref{thm:BlockSubgroups} gives us conditions under which all finitely generated subgroups of a self-replicating branch group are virtually block subgroups. The aim of this section is to obtain conditions under which one can obtain the stronger conclusion that all finitely generated subgroups are virtually \emph{regular} block subgroups (see Definition~\ref{defn:RegularBlock}).

We begin with a lemma about diagonal block subgroups.

\begin{lemma}\label{lemma:DiagonalsSendRistToRist}
Let $G\leq \Aut(X^*)$ be a self-replicating branch group acting tree-primitively on $X^*$, let $H\leq G$ be a finitely generated subgroup and let $F\subseteq X^*$ be a full supporting set for $H$. Let $D\leq H$ be a diagonal block subgroup with supporting vertex set $V\subseteq X^*$. Let us further suppose that $D$ is normalised by $\St_H(V)$ and that $\varphi_v(\St_H(V))=G$ for all $v\in V$. Then, for every $v,w\in V$, the map
\begin{align*}
\alpha_{vw}\colon \varphi_v(D) &\rightarrow \varphi_w(D)\\
\varphi_v(g) &\mapsto \varphi_w(g)
\end{align*}
is a well-defined group isomorphism such that, for every $u\in X^*$, there exists $u'\in X^{|u|}$ such that $\alpha_{vw}(\rist_{\varphi_v(D)}(u)) = \rist_{\varphi_w(D)}(u')$.
\end{lemma}
\begin{proof}
The map $\alpha_{vw}$ is a well-defined isomorphism, since $D$ is a diagonal block subgroup with supporting vertex set $V$, so that the maps $\varphi_v$ and $\varphi_w$ are isomorphisms onto their images. Now, let $u\in X^*$ be any element. Since, by hypothesis, $\St_H(V)$ normalises $D$, the subgroup $\rist_{\varphi_v(D)}(u)$ must be almost normal in $\varphi_v(\St_H(V))=G$. Consequently, the subgroup $\alpha_{vw}(\rist_{\varphi_v(D)}(u))$ must also be an almost normal subgroup of $\varphi_w(\St_H(V))=G$, where we use here the fact that any conjugate of $\alpha_{vw}(\rist_{\varphi_v(D)}(u))$ by an element of $G$ must be the image by $\alpha_{vw}$ of a conjugate of $\rist_{\varphi_v(D)}(u)$ by a (possibly different) element of $G$.

By Theorem~\ref{thm:StructureAlmostNormalSubgroups}, there exist $n\in \N$ and $V'\subseteq X^n$ such that we have $\rist_G(n)'\cap \rist_G(V')\leq \alpha_{vw}(\rist_{\varphi_v(D)}(u)) \leq \rist_G(V')$. Since the conjugates of $\rist_{\varphi_v(D)}(u)$ by elements of $G$ either coincide or commute, with exactly $|X|^{|u|}$ distinct conjugates, the same must be true of $\alpha_{vw}(\rist_{\varphi_v(D)}(u))$, using once again the fact that conjugates of $\alpha_{vw}(\rist_{\varphi_v(D)}(u))$ are in correspondence with the conjugates of $\rist_{\varphi_v(D)}(u)$. This implies that the translates of $V'$ under the action of $G$ are either equal to $V'$ or disjoint, with exactly $|X|^{|u|}$ different translates. Using the fact that $G$ acts transitively on $X^n$, we can conclude that the $|X|^{|u|}$ distinct translates of $V'$ form a partition of $X^n$. Since the action of $G$ on $X^*$ is tree-primitive, this partition must in fact be equal to the partition $\{zX^{n-|u|}\}_{z\in X^{|u|}}$, which means that there exists $u'\in X^{|u|}$ such that $\alpha_{vw}(\rist_{\varphi_v(D)}(u)) \leq \rist_{\varphi_w(D)}(u')$.

Symmetrically, one obtains that $\alpha_{wv}(\rist_{\varphi_w(D)}(u')) \leq \rist_{\varphi_v(D)}(u'')$ for some $u''\in X^{|u|}$. However, it is obvious from their definitions that the maps $\alpha_{vw}$ and $\alpha_{wv}$ are inverses of one another, so that $u''=u$. We thus have $\alpha_{vw}(\rist_{\varphi_v(D)}(u)) \leq \rist_{\varphi_w(D)}(u') \leq \alpha_{vw}(\rist_{\varphi_v(D)}(u))$, which concludes the proof.
\end{proof}

Using the previous lemma, we shall see that under good circumstances, one can decompose diagonal block subgroups as products of smaller diagonal block subgroups.

\begin{lemma}\label{lemma:FromDiagonalToBlockWithGoodProjection}
Let $G\leq \Aut(X^*)$ be a self-replicating branch group acting tree-primitively on $X^*$, let $H\leq G$ be a finitely generated subgroup and let $F\subseteq X^*$ be a full supporting set for $H$. Let $D\leq H$ be a diagonal block subgroup with supporting vertex set $V\subseteq X^*$ such that $D$ is normalised by $\St_H(V)$ and $\varphi_v(\St_H(V))=G$ for all $v\in V$. For every $n\in \N$, there exists a subgroup $D_n\leq D$ of finite index in $D$ and a partition $P=\{V_i\}_{i=1}^{|X|^n}$ of $VX^{n}$ such that $D_n$ is a block subgroup with supporting partition $P$ and for all $v\in V$, $z\in X^n$,
\[\varphi_{vz}(D_n) = \varphi_z(\rist_{\varphi_v(D)}(z)).\]
\end{lemma}
\begin{proof}
Let us fix $v\in V$. By Lemma~\ref{lemma:DiagonalsSendRistToRist}, for every $w\in V$ and for every $x\in X$, there exists an element $y(w,x)\in X$ such that $\alpha_{vw}(\rist_{\varphi_v(D)}(x))=\rist_{\varphi_w(D)}(y(w,x))$, where $\alpha_{vw}\colon \varphi_v(D) \rightarrow \varphi_w(D)$ is the isomorphism described in Lemma~\ref{lemma:DiagonalsSendRistToRist}. As $\alpha_{vw}$ is an isomorphism, it is clear that if $x,x'\in X$ are such that $x\ne x'$, then $y(w,x)\ne y(w,x')$, which means that the map $y(w,\cdot)\colon X \rightarrow X$ is a bijection for all $w\in V$. Thus, if for $x\in X$ we define $V_x=\{wy(w,x)\in X^* \mid w\in V\}$, the collection $P=\{V_x\}_{x\in X}$ forms a partition of $VX$.

For every $x\in X$, let us define $D_x=\{g\in D \mid \varphi_v(g)\in \rist_{\varphi_v(D)}(x)\}$. We want to show that $D_x$ is a diagonal block subgroup with supporting vertex set $V_x$. For this, we first need to show that $D_x\leq \Srist_G(V_x)$. Since $D_x\leq D$, we already have $D_x\leq \St_G(V)$. Now, let $w\in V$ be any element. We have
\begin{align*}
\varphi_w(D_x)=\alpha_{vw}(\varphi_v(D_x)) =\alpha_{vw}(\rist_{\varphi_v(D)}(x)) = \rist_{\varphi_w(D)}(y(w,x)), \tag{\textasteriskcentered} \label{eq:PhiDxIsRist}
\end{align*}
which implies that $D_x\leq \St_G(VX)$ and that $\varphi_{z}=1$ for all $z\notin V_x$. Thus, we have $D_x\leq \Srist_G(V_x)$.

The second point we need to verify to show that $D_x$ is a diagonal block subgroup is that $\varphi_z(D_x)$ is of finite index in $G$ for all $z\in V_x$. By assumption, $\varphi_w(D)$ is of finite index in $G$ for all $w\in V$. It follows that for all $x\in X$, $\rist_{\varphi_w(D)}(x)$ is of finite index in $\rist_G(x)$ and thus that $\varphi_x(\rist_{\varphi_w(D)}(x))$ is of finite index in $\varphi_x(\rist_G(x))$. Since $G$ is a self-replicating branch group, $\varphi_x(\rist_G(x))$ is of finite index in $G$, which means that $\varphi_x(\rist_{\varphi_w(D)}(x))$ is of finite index in $G$. Let $z\in V_x$ be any element. Then, there exists $w\in V$ such that $z=wy(w,x)$. By the equation \eqref{eq:PhiDxIsRist} above, we have
\[\varphi_{z}(D_x) = \varphi_{y(w,x)}(\varphi_w(D_x)) = \varphi_{y(w,x)}(\rist_{\varphi_w(D)}(y(w,x))),\]
which is of finite index in $G$.

To finish showing that $D_x$ is a diagonal block subgroup, we still need to verify that $\varphi_z$ is injective on $D_x$ for all $z\in V_x$. Let $z=wy(w,x)$ be any element of $V_x$, and let $g\in D_x$ be such that $\varphi_z(g)=1$. Then, $\varphi_{y(w,x)}(\varphi_w(g))=1$. Since $\varphi_w(g)\in \rist_{\varphi_w(D)}(y(w,x))$, this implies that $\varphi_w(g)=1$, which in turns implies that $g=1$, as $\varphi_w$ is injective on $D$ by assumption.

We have thus shown that for every $x\in X$, the subgroup $D_x$ is a diagonal block subgroup with supporting vertex set $V_x$. Let $D_1=\prod_{x\in X}D_x$. By construction, $D_1$ is a block subgroup with supporting partition $P=\{V_x\}_{x\in X}$ such that for all $w\in V$ and all $x\in X$,
\[\varphi_{wx}(D_1) = \varphi_x(\rist_{\varphi_w(D)}(x)).\]
To finish proving the result for $n=1$, it thus only remains to show that $D_1$ is of finite index in $D$. As $\varphi_v$ is injective on $D$ for every $v\in V$, this is equivalent to showing that $\varphi_v(D_1)$ is of finite index in $\varphi_v(D)$, and since $\varphi_v(D)$ is of finite index in $G$, this is equivalent to showing that $\varphi_v(D_1)$ is of finite index in $G$. By construction, $\varphi_v(D_1) = \prod_{x\in X}\rist_{\varphi_{v}(D)}(x)$. We have seen above that $\rist_{\varphi_v(D)}(x)$ is of finite index in $\rist_G(x)$. Therefore, $\prod_{x\in X}\rist_{\varphi_{v}(D)}(x)$ is of finite index in $\rist_G(1)$, which is of finite index in $G$, since $G$ is a branch group.

We have just shown that the result holds for $n=1$, which is sufficient to show that it holds for all $n\in \N$, using induction. Indeed, suppose that the result holds for some $n\in \N$ and let us show that it holds for $n+1$. By assumption, there exists $D_n\leq D$ of finite index and a partition $P=\{V_i\}_{i=1}^{|X|^{n}}$ of $VX^{n}$ such that $D_n$ is a block subgroup with supporting partition $P$. For every $1\leq i \leq |X|^{n}$, let $D_{V_i}\leq D_n$ be the diagonal block subgroup with supporting vertex set $V_i$. By our argument for $n=1$, one finds a partition $P_{V_i}=\{W_{i,j}\}_{j=1}^{|X|}$ of $V_iX$ and a finite index subgroup $D_{V_i,1}\leq D_{V_i}$ such that $D_{V_i,1}$ is a block subgroup with supporting partition $P_{V_i}$. One can then define $D_{n+1}$ as the product of the $D_{V_i,1}$. It is then immediate to check that $D_{n+1}$ is a block subgroup for the partition $\{W_{i,j}\mid 1\leq i \leq |X|^n, 1\leq j \leq |X|\}$ and is of finite index in $D_n$, and thus in $D$. To finish the proof, we then only need to remark that for all $v\in V$, $z\in X^{n}$ and $x\in X$, we have by construction
\[\varphi_{vzx}(D_{n+1}) = \varphi_x(\rist_{\varphi_{vz}(D_n)}(x)),\]
and since $\varphi_{vz}(D_n) = \varphi_z(\rist_{\varphi_v(D)}(z))$, we have
\begin{align*}
\varphi_{vzx}(D_{n+1}) &= \varphi_x(\rist_{\varphi_z(\rist_{\varphi_v(D)}(z))}(x))\\
&=\varphi_x(\varphi_z(\rist_{\varphi_v(D)}(zx)))\\
&=\varphi_{zx}(\rist_{\varphi_{v}(D)}(zx)),
\end{align*}
which shows that $D_{n+1}$ has the required property.
\end{proof}

We are now in position to prove that, under certain conditions, all finitely generated subgroups of a self-replicating branch groups with the subgroup induction property are virtually regular block subgroups.

\begin{thm}\label{thm:RegularBlockSubgroup}
Let $G\leq \Aut(X^*)$ be a finitely generated self-replicating regular branch group with trivial branch kernel (see Definition~\ref{defn:TrivialBranchKernel}) and the subgroup induction property acting tree-primitively on $X^*$.
Let $K$ be its maximal branching subgroup (see Definition~\ref{defn:MaxBranchingSubgroup} and Corollary~\ref{cor:MaximalBranchingSubgroup}).
Then, every finitely generated subgroup $H\leq G$ is virtually a regular block subgroup over $K$ (i.e., it admits a finite index subgroup that is a regular block subgroup over $K$).
\end{thm}
\begin{proof}
Let $H\leq G$ be a finitely generated subgroup of $G$. By Corollary~\ref{cor:BlockSubgroupNiceProjections}, there exists a full supporting set $F\subseteq X^*$ for $H$ and a partition $P=\{F_i\}_{i=1}^{k}$ of $F$ such that $D=D_1\cdots D_k\leq H$ is a block subgroup of finite index with supporting partition $P$, where $D_i=\Srist_H(F_i)$, and $\varphi_v(\St_H(F_i))=G$ for all $i\in \{1,\dots, k\}$ and all $v\in F_i$.

By definition of a block subgroup, $\varphi_v(D)$ is of finite index in $G$ for all $v\in F$. Since we assumed that $G$ has trivial branch kernel, this means that for every $v\in F$, there exists $n_v\in \N$ such that $\rist_G(n_v)\leq \varphi_v(D)$. Therefore, for every $w\in X^m$ with $m\geq n_v$, we have $\rist_{\varphi_v(D)}(w)=\rist_G(w)$.
Since $K$ is the maximal branching subgroup, by Corollary~\ref{cor:MaximalBranchingSubgroup} there exists $n_0\in \N$ such that $K=\varphi_{v}(\rist_G(v))$ for any $v\in X^{n}$ with $n\geq n_0$.
Let $n_1=\max\{n_0, \max_{v\in F}\{n_v\}\}$. It follows that for all $w\in X^{n_1}$ and all $v\in F$, we have
\[\varphi_{w}(\rist_{\varphi_{v}(D)}(w)) = \varphi_w(\rist_G(w)) = K.\]

Notice that each $D_i=\Srist_H(F_i)$ is normalised by $\St_H(F_i)$. Therefore, by Lemma~\ref{lemma:FromDiagonalToBlockWithGoodProjection}, for every $D_i$, there exists a subgroup $B_i\leq D_i$ of finite index and a partition $P_i$ of $F_iX^{n_1}$ such that $B_i$ is a block subgroup with supporting partition $P$ and such that for all $v\in F_i$ and all $z\in X^{n_1}$,
\[\varphi_{vz}(B_i) = \varphi_z(\rist_{\varphi_v(D_i)}(z)) = \varphi_z(\rist_{\varphi_v(D)}(z)) = K,\]
where we used the fact that, as $D$ is a block subgroup, $\varphi_v(D_i) = \varphi_v(D)$ for any $1\leq i \leq k$ and $v\in F_i$. In other words, each $B_i$ is a regular block subgroup over $K$ with supporting partition $P_i$. As the sets $F_iX^{n_1}$ are pairwise disjoint, the union of the partitions $P_i$ gives us a partition $P'=\bigcup_{i=1}^{k}P_i$ of $F_iX^{n_1}$, which is a set of pairwise incomparable elements. Since each $B_i$ is a regular block subgroup over $K$ with supporting partition $P_i$, it is immediate to check that their product $B=\prod_{i=1}^{k}B_i$ is also a regular block subgroup over $K$, with supporting partition $P'$. Furthermore, as each $B_i$ is a finite index subgroup of $D_i$, $B$ must be a finite index subgroup of $D$ and thus of $H$.
\end{proof}

As an application, we obtain that all finitely generated subgroups of the Grigorchuk group and of the torsion GGS $p$-groups are virtually regular block subgroups.

\begin{cor}
Every finitely generated subgroup of  $\Grig=\langle a,b,c,d \rangle$ (see Example~\ref{example:GrigorchukGroup}) admits a regular block structure over $K=\langle [a,b] \rangle_{\Grig}$.
\end{cor}
\begin{proof}
As discussed in Example~\ref{example:GrigorchukGroup}, $\Grig$ is a self-replicating branch group over $K$.
By~\cite[Theorem 7.9]{BartholdiGrigorchukSunic03}, $K$ is the maximal branching subgroup.
Furthermore, $\Grig$ has the subgroup induction property~\cite{GrigorchukWilson03} and a trivial branch kernel, as follows from~\cite[Theorem VIII.42]{delaHarpe00}.
Lastly, the action of $\Grig$ on $X^*$ is tree-primitive by Corollary~\ref{cor:GrigTreePrimitivity}.
The result then follows directly from Theorem~\ref{thm:RegularBlockSubgroup}.
\end{proof}
\begin{cor}
Let $p$ be a prime number and $G=G_E=\langle a,b\rangle \leq\Aut(X^*)$ be a torsion GGS-group for $X=\{0,\dots, p-1\}$ (see Example~\ref{example:GGS}).
Then every finitely generated subgroup of $G$ admits a regular block structure over $\gamma_3(G)$ if $E$ is symmetric, or over $G'$ if $E$ is not symmetric.
\end{cor}
\begin{proof}
Let $K$ be either $\gamma_3(G)$ if $E$ is symmetric or $G'$ is $E$ is not symmetric. We want to prove that $K$ is the maximal branching subgroup.
Since $p$ is prime and $G$ torsion, $E$ is a non-constant vector.
It follows, as discussed in Example~\ref{example:GGS}, that $G$ is a self-replicating branch group over $\gamma_3(G)$.
Moreover, $G$ is regular branch over $G'$ if and only if $E$ is not symmetric,~\cite[Lemma 3.4]{FAZR14} and~\cite[Proof of Theorem 3.7]{FAZR14}.
We claim that if a GGS-group $G$ is regular branch over a subgroup $L$, then $L$ is contained in $G'$.
Since $\gamma_3(G)$ has index $p$ in $G'$, this implies that $K$ is the maximal branching subgroup.

For $i\in\{0,\dots,p-1\}$, let $b_i\coloneq a^iba^{-i}$.
Let $l$ be an element of $L$. As $G/G'$ is isomorphic to $\langle a\rangle\times\langle b\rangle$, we have $l=a^nb^mh$ for some $h\in G'$ and we want to show that $n=m=0$.
By regular branchness, the element $g=(l,1,\dots,1)$ is in the intersection of $L$ and of $\St_G(1)=\langle b_0,\dots, b_{p-1}\rangle$.
By~\cite[Theorem 2.10]{FAZR14}, for $i\in\{0,\dots,p-1\}$, the number $w_i$ of $b_{i}$ appearing in any word in $\{b_0,\dots, b_{p-1}\}$ representing $g$ is well-defined modulo~$p$. Moreover, for $i\neq 1$ and modulo $p$, $w_i$ is also the number of $b=b_0$ appearing in any word representing $1$ (in the $i$-th position in $g=(l,1,\dots,1)$). We conclude that for $i\neq 1$ we have $w_i\equiv 0 \pmod p$.
But then, the number of $a$ in any word representing $\varphi_1(g)=l$ is equal to $\sum_{i\neq 1}w_ie_i\equiv 0\pmod p$, which implies $n=0$.
We also have that $m=0$. Indeed, otherwise we would have $w_1\not\equiv0$ modulo $p$. But then for $j\neq 1$ we would have $\varphi_j(g)\neq 1$ as soon as $e_j\neq 0$, which is absurd.

Furthermore, $G$ has the subgroup induction property~\cite{FrancoeurLeemann20} and a trivial branch kernel, as follows from~\cite[Theorem 2.6]{FAGUA17}.
Lastly, the action of $G$ on $X^*$ is tree-primitive by Corollary~\ref{cor:GGSTreePrimitivity}.
The result then follows directly from Theorem~\ref{thm:RegularBlockSubgroup}.
\end{proof}

\bibliography{FGLNBib}

\providecommand{\noopsort}[1]{} \def\cprime{$'$}
\begin{thebibliography}{DDFAG23}

\bibitem[BG{\v{S}}03]{BartholdiGrigorchukSunic03}
Laurent Bartholdi, Rostislav~I. Grigorchuk, and Zoran {\v{S}}uni{\'k}.
\newblock Branch groups.
\newblock In {\em Handbook of algebra, {V}ol. 3}, pages 989--1112.
  North-Holland, Amsterdam, 2003.

\bibitem[Bon10]{MR2727305}
Ievgen~V. Bondarenko.
\newblock Finite generation of iterated wreath products.
\newblock {\em Arch. Math. (Basel)}, 95(4):301--308, 2010.

\bibitem[BRLN16]{MR3478865}
Khalid Bou-Rabee, Paul-Henry Leemann, and Tatiana Nagnibeda.
\newblock Weakly maximal subgroups in regular branch groups.
\newblock {\em J. Algebra}, 455:347--357, 2016.

\bibitem[BSZ12]{BartholdiSiegenthalerZalesskii12}
Laurent Bartholdi, Olivier Siegenthaler, and Pavel Zalesskii.
\newblock The congruence subgroup problem for branch groups.
\newblock {\em Israel J. Math.}, 187:419--450, 2012.

\bibitem[CLB23]{CapraceLeBoudec23}
Pierre-Emmanuel Caprace and Adrien Le~Boudec.
\newblock Commensurated subgroups and micro-supported actions (with an appendix
  by {Dominik} {Francoeur}).
\newblock {\em J. Eur. Math. Soc. (JEMS)}, 25(6):2251--2294, 2023.

\bibitem[Cou01]{Coulbois}
Thierry Coulbois.
\newblock Free product, profinite topology and finitely generated subgroups.
\newblock {\em Internat. J. Algebra Comput.}, 11(2):171--184, 2001.

\bibitem[DDFAG23]{DDFAG23}
E.~Di~Domenico, G.~A. Fern\'{a}ndez-Alcober, and N.~Gavioli.
\newblock G{GS}-groups over primary trees: branch structures.
\newblock {\em Monatsh. Math.}, 200(4):781--797, 2023.

\bibitem[dlH00]{delaHarpe00}
Pierre de~la Harpe.
\newblock {\em Topics in geometric group theory}.
\newblock Chicago Lectures in Mathematics. University of Chicago Press,
  Chicago, IL, 2000.

\bibitem[FAGUA17]{FAGUA17}
Gustavo~A. Fernández-Alcober, Alejandra Garrido, and Jone Uria-Albizuri.
\newblock On the congruence subgroup property for {GGS}-groups.
\newblock {\em Proc. Amer. Math. Soc.}, 145(8):3311--3322, 2017.

\bibitem[FAZR14]{FAZR14}
Gustavo~A. Fernández-Alcober and Amaia Zugadi-Reizabal.
\newblock {GGS}-groups: order of congruence quotients and {H}ausdorff
  dimension.
\newblock {\em Trans. Amer. Math. Soc.}, 366(4):1993--2017, 2014.

\bibitem[FG18]{MR3886188}
Dominik Francoeur and Alejandra Garrido.
\newblock Maximal subgroups of groups of intermediate growth.
\newblock {\em Adv. Math.}, 340:1067--1107, 2018.

\bibitem[FL25]{FrancoeurLeemann20}
Dominik Francoeur and Paul-Henry Leemann.
\newblock Subgroup induction property for branch groups.
\newblock {\em J. Fractal Geom.}, 12(1-2):175--207, 2025.

\bibitem[Fra20]{Francoeur2020}
Dominik Francoeur.
\newblock On maximal subgroups of infinite index in branch and weakly branch
  groups.
\newblock {\em J. Algebra}, 560:818--851, 2020.

\bibitem[Fra21]{Francoeur22p}
Dominik Francoeur.
\newblock On quasi-2-transitive actions of branch groups, 2021.
\newblock {\em arXiv e-prints},arXiv:2111.11967.

\bibitem[Gar16]{Garrido16}
Alejandra Garrido.
\newblock Abstract commensurability and the {G}upta-{S}idki group.
\newblock {\em Groups Geom. Dyn.}, 10(2):523--543, 2016.

\bibitem[GLN21]{GrigorchukLeemannNagnibeda21}
Rostislav~I. Grigorchuk, P.-H. Leemann, and T.~V. Nagnibeda.
\newblock Finitely generated subgroups of branch groups and subdirect products
  of just infinite groups.
\newblock {\em Izv. Ross. Akad. Nauk Ser. Mat.}, 85(6):104--125, 2021.

\bibitem[GN08]{GN-Oberwolfach}
Rostistlav~I. Grigorchuk and Tatiana Nagnibeda.
\newblock On subgroup structure of a 3-generated 2-group of intermediate
  growth.
\newblock In {\em Profinite and Asymptotic Group Theory (MFO, Oberwolfach,
  Germany, June 2008)}, volume 5 (2), pages 1568--1571, Oberwolfach, Germany,
  2008. Mathematisches Forschungsinstitut Oberwolfach.
\newblock Abstracts from the workshop held June 22--28, 2008, Organized by
  Fritz Grunewald and Dan Segal, Oberwolfach Reports. Vol. 5, no. 2.

\bibitem[Gri80]{Grigorchuk80}
Rostislav~I. Grigorchuk.
\newblock On {B}urnside's problem on periodic groups.
\newblock {\em Funktsional. Anal. i Prilozhen.}, 14(1):53--54, 1980.

\bibitem[Gri00]{Grigorchuk00}
Rostislav~I. Grigorchuk.
\newblock Just infinite branch groups.
\newblock In {\em New horizons in pro-{$p$} groups}, volume 184 of {\em Progr.
  Math.}, pages 121--179. Birkhäuser Boston, Boston, MA, 2000.

\bibitem[Gri11]{MR2893544}
Rostislav~I. Grigorchuk.
\newblock Some problems of the dynamics of group actions on rooted trees.
\newblock {\em Tr. Mat. Inst. Steklova}, 273(Sovremennye Problemy
  Matematiki):72--191, 2011.

\bibitem[G{\v{S}}23]{GarridoSunic23}
Alejandra Garrido and Zoran {\v{S}}uni{\'c}.
\newblock Branch groups with infinite rigid kernel, 2023.
\newblock {\em arXiv e-prints},arXiv:2310.06581.

\bibitem[GW03a]{GrigorchukWilson03}
Rostislav~I. Grigorchuk and John~S. Wilson.
\newblock A structural property concerning abstract commensurability of
  subgroups.
\newblock {\em J. London Math. Soc. (2)}, 68(3):671--682, 2003.

\bibitem[GW03b]{GrigorchukWilson03b}
Rostislav~I. Grigorchuk and John~S. Wilson.
\newblock The uniqueness of the actions of certain branch groups on rooted
  trees.
\newblock {\em Geom. Dedicata}, 100:103--116, 2003.

\bibitem[GW14]{GarridoWilson14}
Alejandra Garrido and John~S. Wilson.
\newblock On subgroups of finite index in branch groups.
\newblock {\em J. Algebra}, 397:32--38, 2014.

\bibitem[Lee24]{Leemann20}
Paul-Henry Leemann.
\newblock Weakly maximal subgroups of branch groups, 2024.
\newblock {\em arXiv e-prints},arXiv:1910.06399.

\bibitem[Mal83]{Malcev}
A.~I. Mal'tsev.
\newblock On homomorphisms onto finite groups.
\newblock {\em Transl., Ser. 2, Am. Math. Soc.}, 119:67--79, 1983.
\newblock Russian original in Ivanov. Gos. Ped. Inst. Ucen. Zap., Vol. 18, pp.
  49--60 (1956).

\bibitem[MM22]{MM}
Ashot Minasyan and Lawk Mineh.
\newblock Quasiconvexity of virtual joins and separability of products in
  relatively hyperbolic groups, 2022.
\newblock {\em arXiv e-prints},arXiv:2207.03362.

\bibitem[Per00]{MR1841763}
Ekaterina~L. Pervova.
\newblock Everywhere dense subgroups of a group of tree automorphisms.
\newblock {\em Tr. Mat. Inst. Steklova}, 231(Din. Sist., Avtom. i Beskon.
  Gruppy):356--367, 2000.

\bibitem[Rub96]{Rubin96}
Matatyahu Rubin.
\newblock Locally moving groups and reconstruction problems.
\newblock In {\em Ordered groups and infinite permutation groups}, volume 354
  of {\em Math. Appl.}, pages 121--157. Kluwer Acad. Publ., Dordrecht, 1996.

\bibitem[RZ93]{MR1190361}
Luis Ribes and Pavel~A. Zalesskii.
\newblock On the profinite topology on a free group.
\newblock {\em Bull. London Math. Soc.}, 25(1):37--43, 1993.

\bibitem[UA16]{UriaAlbizuri16}
Jone Uria-Albizuri.
\newblock On the concept of fractality for groups of automorphisms of a regular
  rooted tree.
\newblock {\em Reports@ SCM}, 2(1):33--44, 2016.
\newblock http://revistes.iec.cat/index.php/reports/article/view/139172.

\bibitem[Vov00]{Vovkivsky98}
Taras Vovkivsky.
\newblock Infinite torsion groups arising as generalizations of the second
  {G}rigorchuk group.
\newblock In {\em Algebra. Proceedings of the international algebraic
  conference on the occasion of the 90th birthday of A. G. Kurosh, Moscow,
  Russia, May 25--30, 1998}, pages 357--377. de Gruyter, Berlin, 2000.

\bibitem[Wil71]{Wilson}
John~S. Wilson.
\newblock Groups with every proper quotient finite.
\newblock {\em Proc. Cambridge Philos. Soc.}, 69:373--391, 1971.

\bibitem[Wil08]{MR2399104}
Henry Wilton.
\newblock Hall's theorem for limit groups.
\newblock {\em Geom. Funct. Anal.}, 18(1):271--303, 2008.

\end{thebibliography}
\bibliographystyle{alpha}

\end{document}